\documentclass[a4paper,10pt]{amsart}
\usepackage{amsmath,amssymb,amsthm,leftidx}

\newcommand{\supp}{\operatorname{supp}}

\renewcommand{\>}{\rangle}
\newcommand{\<}{\langle}

\newcommand{\IM}{\operatorname{\mathcal{M}}}

\newcommand{\LL}{\operatorname{\mathcal{L}}}

\newcommand{\IS}{\operatorname{\mathcal{S}}}

\newcommand{\CP}{\mathbb{C}\mathbb{P}}
\newcommand{\eps}{\epsilon}

\newcommand{\CR}{\bar{\partial}}

\newcommand{\PD}{\operatorname{PD}}

\newcommand{\loc}{\operatorname{loc}}

\newcommand{\CZ}{\operatorname{CZ}}

\newcommand{\HF}{\operatorname{HF}}

\newcommand{\del}{\partial}
\newcommand{\RS}{\IR \times S^1}

\newcommand{\IC}{\operatorname{\mathbb{C}}}
\newcommand{\IZ}{\operatorname{\mathbb{Z}}}

\newcommand{\IR}{\operatorname{\mathbb{R}}}
\newcommand{\IN}{\operatorname{\mathbb{N}}}
\newcommand{\IH}{\operatorname{\mathbb{H}}}

\newcommand{\IP}{\operatorname{\mathbb{P}}}
\renewcommand{\IS}{\operatorname{\mathbb{S}}}

\renewcommand{\LL}{\operatorname{\mathcal{L}}}

\newtheorem{theorem}{Theorem}[section]
\newtheorem{proposition}[theorem]{Proposition}

\newtheorem{definition}[theorem]{Definition}
\newtheorem{lemma}[theorem]{Lemma}

\newtheorem{remark}[theorem]{Remark}

\title{Hamiltonian Floer theory for\\ nonlinear Schr\"odinger equations\\ and the small divisor problem}
\author{Oliver Fabert}
\thanks{O. Fabert, VU Amsterdam, The Netherlands. Email: o.fabert@vu.nl}
\begin{document}
\maketitle

\begin{abstract}
We prove the existence of infinitely many time-periodic solutions of nonlinear Schr\"odinger equations using pseudo-holomorphic curve methods from Hamiltonian Floer theory. For the generalization of the Gromov-Floer compactness theorem to infinite dimensions, we show how to solve the arising small divisor problem by combining elliptic methods with results from the theory of diophantine approximations.   
\end{abstract}

\tableofcontents
\markboth{O. Fabert}{Floer theory for Hamiltonian PDE} 

\section{Hamiltonian partial differential equations}

Nonlinear Schr\"odinger equations play a very important role in mathematical physics and have applications in, e.g., solid
state physics, condensed matter physics, quantum chemistry, nonlinear optics, wave propagation, protein folding and the semiconductor industry. They are classical field equations describing multi-particle systems, where the nonlinearity models the interaction between different particles. Before studying nonlinear examples, recall that the well-known linear free Schr\"odinger equation, i.e., without exterior potential, is given by $$i\del_t u \,=\, - \Delta u,$$ where $u=u(t,x)\in\IC$ is a complex-valued function depending on time and space, $\del_t$ is the derivative with respect to the time $t\in\IR$ and $\Delta=\del_x^2$ denotes the Laplace operator with respect to the space coordinate $x$. Here and in what follows we restrict ourself to the case of one spatial dimension. \\

Nonlinear Schr\"odinger equations are important examples of \emph{Hamiltonian partial differential equations}, where we refer to \cite{K} for definitions, statements and further references. This means that they can be written in the form $\del_t u=X^H_t(u)$, where the Hamiltonian vector field $X^H_t$ is determined by the choice of a (time-dependent) Hamiltonian function $H=H_t$ and a linear symplectic form $\omega$. Here a bilinear form $\omega:\IH\times\IH\to\IR$ on a real Hilbert space $\IH$ is called symplectic if it is anti-symmetric and nondegenerate in the sense that the induced linear mapping $i_{\omega}:\IH\to\IH^*$ is an isomorphism. As in the finite-dimensional case it can be shown that for any symplectic form $\omega$ there exists a complex structure $J_0$ on $\IH$ such that $\omega$, $J_0$ and the real inner product $\langle\cdot,\cdot\rangle_{\IR}$ on $\IH$ are related via $\langle\cdot,\cdot\rangle_{\IR}=\omega(\cdot,J_0\cdot)$.  \\

In the case of nonlinear Schr\"odinger equations on the circle $S^1=\IR/2\pi\IZ$ one chooses the complex Hilbert space $\IH=L^2(S^1,\IC)$ of square-integrable complex-valued functions on the circle which naturally can be viewed as a real Hilbert space by identifying $\IC$ with $\IR^2$. The standard complex inner product $\<\cdot,\cdot\>_{\IC}$ is related to the standard real inner product $\<\cdot,\cdot\>_{\IR}$ and the standard symplectic form $\omega$ by $\<\cdot,\cdot\>_{\IC}=\<\cdot,\cdot\>_{\IR}+i\omega$ and the symplectic form is related to the real inner product via $\omega=\<J_0\cdot,\cdot\>_{\IR}$ with $J_0=i$ denoting the standard complex structure on $\IH$. Furthermore we denote the resulting $L^2$-norm by $|\cdot|=|\cdot|_2$, where we usually omit the subscript. In order to stress the relation with the finite-dimensional case of $\IR^{2n}=\IC^n$, note that, using the Fourier series expansion $u(x)=(2\pi)^{-1/2}\sum_{n=-\infty}^{\infty} \hat{u}(n)\cdot \exp(inx)$ and writing $\hat{u}(n)=q_n+ip_n$ for all $n\in\IZ$, it follows that the symplectic Hilbert space $L^2(S^1,\IC)$ can be identified with the space $\ell^2(\IC)=\ell^2(\IZ,\IC)$ of square-summable complex-valued series $\hat{u}:\IZ\to\IC$ equipped with the symplectic form $\omega=\sum_{n=-\infty}^{+\infty} dp_n\wedge dq_n$. We want to stress that classical dense subspaces like $\IH_{\delta}:=H^{\delta,2}(S^1,\IC)=\{u\in\IH: \sum_{n=-\infty}^{\infty} |\hat{u}(n)|^2\cdot |n|^{2\delta}<\infty\}$ for $\delta>0$ and hence also $C^{\infty}(S^1,\IC)=\IH_{\infty}:=\bigcap_{\delta>0}\IH_{\delta}$ are only weakly symplectic in the sense that $i_{\omega}:\IH_{\delta}\to\IH_{\delta}^*$ is only injective; in particular, real-valued functions on $\IH_\infty$ which are even smooth typically do not possess a Hamiltonian vector field. For this observe that $i_{\omega}$ restricts to an isomorphism between $\IH_{\delta}$ and $\IH_{-\delta}^*$ which is a proper subspace of $\IH_{\delta}^*$ for all $\delta>0$.  

\section{Nonlinear Schr\"odinger equations of convolution type}

While the symplectic form $\omega$ is nondegenerate on $L^2(S^1,\IC)$, the Hamiltonian $H_t=H^0$ given by $$H^0(u)\,=\,-\int_0^{2\pi}  \frac{|u_x(x)|^2}{2}\;dx\,=\,- \sum_{n=-\infty}^{+\infty} \frac{n^2}{2}|\hat{u}(n)|^2$$ is only well-defined on its dense subspace $H^{1,2}(S^1,\IC)$. While the resulting Hamiltonian vector field $X^0(u)=i\Delta u$ is only defined on $H^{2,2}(S^1,\IC)$, we can prove the following result about the corresponding flow $\phi^0_t$.

\begin{proposition}\label{free-NLS} After applying the Fourier transform, the flow of the free Schr\"odinger equation is given by $$\phi^0_t(\hat{u})(n)=\exp(itn^2)\cdot\hat{u}(n)\,\,\textrm{for all}\,\,n\in\IN.$$ For fixed time $t$ it preserves the $L^2$-norm and hence defines a linear symplectomorphism on the full symplectic Hilbert space $\IH=L^2(S^1,\IC)$, which restricts to a finite-dimensional linear symplectomorphism on every $\IC^{2k+1}:=\{u\in\IH:\,\hat{u}(n)=0\,\,\textrm{for all}\,\, |n|>k\}$. \end{proposition}

\begin{proof} In order to see this, observe that, after applying the Fourier transform, the symplectic vector field is given by $X^0(\hat{u})(n)=in^2\cdot\hat{u}(n)$. On the other hand, since in every frequency the corresponding flow $\phi^0_t$  multiplies the Fourier coefficient by a complex number of norm one, the claims follow. \end{proof}

In this paper we want to restrict ourselves to nonlinearities which can be approximated very well by finite-dimensional ones. More precisely, we want to restrict ourselves to the following modification of the classical nonlinear Schr\"odinger equation, see \cite{K}.

\begin{definition} A  \emph{nonlinear Schr\"odinger equation of convolution type} is a Hamiltonian PDE with Hamiltonian $H_t=H^0+F_t$ on the symplectic Hilbert space $\IH=L^2(S^1,\IC)$, where the Hamiltonian defining the nonlinearity is defined as $$F_t(u)\,:=\, - \int_0^{2\pi} \frac{1}{2} f(|(u*\psi)(x)|^2,x,t)\,dx$$ with $(u*\psi)(x)=\<u,\psi(x-\cdot)\>_{\IC}$, where $\psi\in C^{\infty}(S^1,\IR)$ is some fixed smoothing kernel, and $f$ denotes a smooth, real-valued function on $\IR^+\times S^1\times\IR$. \end{definition}

In other words, for the Hamiltonian $H_t$ we consider density functions of the form $\frac{1}{2}(|u_x|^2+\tilde{f}(u,x,t))$ with $\tilde{f}(u,x,t):=f(|(u*\psi)(x)|^2,x,t)$. Using 
\begin{eqnarray*} 
\<\nabla F_t(u), v\>_{\IR} &=& -\int_0^{2\pi}\frac{1}{2}\frac{d}{ds}|_{s=0} [\tilde{f}(u(x)+sv(x),x,t)-\tilde{f}(u(x),x,t)]\,dx\\ &=& -\int_0^{2\pi}\frac{1}{2}\del_1\tilde{f}(u(x),x,t)\cdot v(x)\,dx
\end{eqnarray*} it follows that the $L^2$-gradient of $F_t$ at $u\in\IH$ is given by $$\nabla F_t(u)(x)=-\frac{1}{2}\del_1\tilde{f}(u,x,t)=-\del_1  f(|u*\psi|^2,x,t)(u*\psi)*\psi$$ which in turn implies that the resulting nonlinear Schr\"odinger equation is given by $$i\del_t u \,=\, - \Delta u \;+\; \del_1  f(|u*\psi|^2,x,t)(u*\psi)*\psi,$$ where $\del_1 f$ means derivative with respect to the first coordinate and convolution is understood with respect to the space coordinate. Nonlinear Schr\"odinger equations of convolution type describe multi-particle systems with non-local interaction.\\

In order to see that these nonlinearities can indeed be approximated by finite-dimensional ones very well, define for the given convolution kernel $\psi$ with Fourier series expansion $\psi(x)=(2\pi)^{-1/2}\sum_{n=-\infty}^{+\infty}\hat{\psi}(n)\exp(inx)$ for each $k\in\IN$ the approximating kernel $$\psi^k(x)\,:=\,\frac{1}{\sqrt{2\pi}}\sum_{n=-k}^k \hat{\psi}(n)\exp(inx)$$ and define the resulting sequence of Hamiltonians $F^k_t$ by $$F^k_t(u)\;:=\;- \frac{1}{2}\int_0^{2\pi} f(|(u*\psi^k)(x)|^2,x,t)\;dx$$ for all $k\in\IN$.   

\begin{lemma}\label{approx} For each $k\in\IN$ the Hamiltonian flow $\phi^{F,k}_t$ of the Hamiltonian $F^k_t$ restricts to a finite-dimensional Hamiltonian flow on $\IC^{2k+1}\subset\IH$. Furthermore the sequence of time-dependent Hamiltonians $F^k_t$ converges with all derivatives to the original Hamiltonian $F_t$ as $k\to\infty$, uniformly on every bounded subset of $\IH$. In particular, the same holds true for the symplectic gradients $X^{F,k}$ and $X^F$. \end{lemma}

\begin{proof} For the first statement it suffices to observe that the symplectic gradient $X^{F,k}_t$ of $F^k_t:\IH\to\IR$ given by $$X^{F,k}_t(u)=- i \del_1 f(|(u*\psi^k)(x)|^2,x,t)(u*\psi^k)*\psi^k$$ has vanishing Fourier coefficients, $\widehat{X^{F,k}_t(u)}(n)=0$, for $|n|>k$, where we use that $\widehat{u*v}(n)=\hat{u}(n)\hat{v}(n)$ and $\widehat{\psi^k}(n)=0$ for $|n|>k$. In order to prove that $F^k_t$ converges with all derivatives to $F_t$ as $k\to\infty$ uniformly on bounded subsets, we first observe that in the $L^2$-norm $|\cdot|=|\cdot|_2$ on $\IH$ we have $|\psi-\psi^k|_2\to 0$ as $k\to\infty$. Defining $\tilde{f}:\IR^+\times S^1\times\IR\to\IR$ by $\tilde{f}(w,x,t)=f(|w|^2,x,t)$, it actually suffices to prove that, for every fixed $t$ and every fixed $\alpha\in\IN$, the $\alpha$.th derivative $x\mapsto \del_1^{\alpha}\tilde{f}((u*\psi^k)(x),x,t)$ converges to $x\mapsto \del_1^{\alpha}\tilde{f}((u*\psi)(x),x,t)$, uniformly in $x$ and $u\in\IH$ with $|u|_2\leq R$ for some fixed $R>0$. For this we observe that $u*\psi^k\to u*\psi$ as $k\to\infty$, uniform in the above sense, since for the supremum norm of $u*\psi - u*\psi^k$ we have by Young's inequality for convolutions $$|u*\psi - u*\psi^k|_{\infty}\leq R\cdot |\psi-\psi^k|_2.$$ Now the result follows using the Lipschitz inequality $$|\del_1^{\alpha}\tilde{f}(v,x,t)-\del_1^{\alpha}\tilde{f}(w,x,t)|\leq L_{\alpha}\cdot |v-w|,$$ with $L_{\alpha}$ denoting the $C^{\alpha+1}$-norm of $\tilde{f}(\cdot,\cdot,t):[0,R|\psi|_2]\times S^1\to\IR$. Note that here we additionally use that $|u*\psi|_{\infty}, |u*\psi^k|_{\infty}\leq R\cdot|\psi|_2$. \end{proof}

Since the smoothing kernel $\psi$ is assumed to be smooth, we actually know that $\psi^k$ converges exponentially fast to $\psi$, that is, $|\hat{\psi}(\pm n)|\cdot n^{\delta} \to 0$ as $n\to\infty$ for every $\delta\in\IN$. Using, in addition, that the Hamiltonian $F_{t+1}=F_t$ depends smoothly on $t$, we get  

\begin{lemma}\label{exp} After restricting to a bounded subset, we indeed we have that $F^k_t$ converges exponentially fast with all derivatives to the original Hamiltonian $F_t$ as $k\to\infty$. Moreover, expanding $\nabla F(u)(t,x)=\nabla F_t(u)(x)$ as a Fourier series with respect to both the space and time coordinates $x$, $t$, $$\nabla F(u)(t,x)\,=\,\frac{1}{\sqrt{2\pi}} \sum_{n,p=-\infty}^{+\infty}\widehat{\nabla F(u)}(n,p)\cdot\exp(inx)(i 2\pi pt), $$ then for every $\alpha,\delta\in\IN$ we have $$\widehat{\nabla F(u)}(n,p)\cdot |n|^{\delta}\cdot |p|^{\alpha}\,\to\, 0\,\,\textrm{as}\,\, |n|,|p|\to\infty,$$ uniformly on every bounded subset of $\IH$. \end{lemma} 

We collect some further important observations about these class of equations in the following  

\begin{proposition} \label{projective} For every nonlinear Schr\"odinger equation of convolution type the resulting flow exists for all times and is given by the composition $\phi_t=\phi^0_t\circ\phi^G_t$ of the flow $\phi^0_t$ of the free Schr\"odinger equation with the flow of the smooth time-dependent Hamiltonian $G_t=F_t\circ\phi^0_t$. Furthermore, since $\phi_t$ preserves the Hilbert space norm on $\IH$, it descends to a symplectic flow on the projective Hilbert space $\IP(\IH)$ equipped with the Fubini-Study form. \end{proposition}

\begin{proof} We first observe that the composition $\phi^0_t\circ\phi^G_t$ is generated by $H^0+G_t\circ\phi^0_{-t}=H^0+F_t=H_t$. Since the claims of the proposition immediately follow when $F_t=0$, it suffices to focus on the flow of $G_t$, or equivalently on the flow of $F_t$. From lemma \ref{approx} it immediately follows that $F$ is smooth; since $\phi^0_t(u)*\psi=\phi^0_t(u*\psi)$ and the free flow $\phi^0_t$ is infinitely often differentiable with respect to time and space coordinates on $C^{\infty}(S^1,\IC)\subset L^2(S^1,\IC)$, the same continues to hold for $G$. Taken additionally proposition \ref{free-NLS} into account, the claims follow after showing that the Hamiltonian flow of $F$ preserves the Hamiltonian $K:\IH\to\IR$ given by (half the square of) the Hilbert space norm, $K(u)=|u|^2/2$: Since $$X^K(F)=-X^F(K)=\omega(X^F,X^K)=\<X^F,\nabla K\>_{\IR}$$ with $\nabla K(u)=u$, it suffices to show that, at every point $u\in\IH$, the symplectic gradient $X^F(u)=X^F_t(u)=i\del_1 f(|(u*\psi)(x)|^2,x,t)(u*\psi)*\psi\in\IH$ is perpendicular to $u$ with respect to the real inner product on $\IH=L^2(S^1,\IR^2)$. First observe that the statement is immediately clear if $u$ is truely real (that is, $u(x)\in\IR\subset\IC$ for all $x\in S^1$) or truely imaginary, as in this case $X^F(u)$ is truely imaginary, or truely real, respectively. For the general case write $u=u_R+i u_I$ and $X^F(u)=X^F_R(u)+i X^F_I(u)$ with real-valued functions $u_R,u_I,X^F_R,X^F_I$. In this case we can show that \begin{eqnarray*} &&\<u_R,\del_1f(|(u*\psi)(x)|^2,x,t)(u_I*\psi)*\psi\>\\&&=\<u_R*\psi,\del_1 f(|(u*\psi)(x)|^2,x,t)(u_I*\psi)\>\\&&=\<\del_1 f(|(u*\psi)(x)|^2,x,t)(u_R*\psi),u_I*\psi\>\\&&=\<\del_1 f(|(u*\psi)(x)|^2,x,t)(u_R*\psi)*\psi,u_I\>, \end{eqnarray*}
which in turn proves that $$\<u,X^F(u)\>\,=\,\<u_I,X^F_R(u)\>\;-\;\<u_R,X^F_I(u)\>\,=\,0,$$ finishing the proof of the proposition. \end{proof}

Here we use the canonical normalization of the Fubini-Study form $\omega$ on $\IP(\IH)$ inherited by viewing $\IP(\IH)$ as the quotient of the unit sphere $\IS(\IH)=\{u\in\IH:|u|_2=1\}$ under the $U(1)$-action; in particular, the symplectic area of every holomorphic sphere in $\IP(\IH)$ is $\pi$. Note that studying the Schr\"odinger equation on $\IP(\IH)$ in place of $\IH$ is also natural from the view point of quantum physics. 

\begin{remark} \label{intrinsic}
The Hamiltonians $F_t$ can be viewed as real-valued functions defined on the weakly symplectic Frechet manifold $(\IP(\IH_\infty),\omega)$ with $\IH_\infty=C^\infty(S^1,\IC)$, which are smooth and have a Hamiltonian vector field $X^F_t(u)\in T_u\IP(\IH_\infty)\subset T_u\IP(\IH)$ in every point  $u\in\IP(\IH_\infty)$ which is \emph{uniformly} bounded (including derivatives) with respect to the canonical sequence of Sobolev metrics defining the Frechet topology on $\IP(\IH_\infty)$. \end{remark}

\section{Statement of the main theorem}

From now on let us assume that the nonlinear term in the Schr\"odinger equation is one-periodic in time, that is, $f(|(u*\psi)(x)|^2,x,t+1)=f(|(u*\psi)(x)|^2,x,t)$. Then it follows that every nonlinear Schr\"odinger equation defines a flow $\phi_t=\phi^H_t$ on the projective Hilbert space $\IP(\IH)$ where the underlying Hamiltonian is one-periodic in time, $H_{t+1}=H^0+F_{t+1}=H^0+F_t=H_t$. \\

In the case of time-one-periodic smooth Hamiltonians on finite-dimensional projective spaces $\CP^n=\IP(\IC^{n+1})$ we have the following famous   

\begin{theorem} (\cite{Fo},\cite{Fl}) The time-one map of a Hamiltonian flow on $\CP^n$ always has at least $n+1$ fixed points, that is, the degenerate version of the famous Arnold conjecture holds. \end{theorem}

Viewing the Hamiltonian flow $\phi_t$ on $\IP(\IH)$ defined by the nonlinear Schr\"odinger equation of convolution type as an infinite-dimensional generalization, it is natural to ask whether an analogue of the Arnold conjecture also holds in this infinite dimensional context, establishing the existence of infinitely many fixed points of the time-one map. \\

But first, in order to show that the generalization to infinite dimensions is not trivial and we can only expect it to hold after imposing restrictions, we first give the following counterexample.

\begin{proposition} There exists a smooth Hamiltonian function on $\IP(\IH)$ whose time-one map has no fixed points at all.  \end{proposition}

\begin{proof} The function $L$ defined by $$L:\IH\to\IR,\,\,L(u) \,:=\, \int_0^{2\pi} V(x) \frac{|u(x)|^2}{2}\;dx$$ decends to a function on the symplectic quotient $\IP(\IH)=\IS(\IH)/S^1$, since its flow map is given by $(\phi^L_t(u))(x) = \exp(it V(x))\cdot u(x)$ and hence preserves the $L^2$-norm. In the same way it can be seen that $u\in\IS(\IH)$ is a fixed point of $\phi^L_1$ on $\IP(\IH)$ if and only if there exists some $a\in\IR$ such that for all $x\in S^1$ we have $\exp(iV(x))\cdot u(x)=\exp(ia)\cdot u(x)$ and hence either $(V(x)-a)/2\pi\in\IZ$ or $u(x)=0$. For a generic choice of the function $V:S^1\to\IR$ it follows that $u(x)=0$ almost everywhere and hence $u=0\not\in\IS(\IH)$, resulting in the fact that its time-one map on $\IP(\IH)$ has no fixed points at all. \end{proof}

Like for the linear Schr\"odinger equation, in our proof the appearance of the Laplace term in the nonlinear Schr\"odinger equations turns out to be essential to find infinitely many fixed points. Before we turn to the general case, we first have a look at the free Schr\"odinger equation where $F_t=0$. The proof of the following proposition is an easy exercise.

\begin{proposition} After passing to the projectivization, the time-one flow map $\phi^0=\phi^0_1$ of the free Schr\"odinger equation on $\IP(\IH)$ has infinitely many different fixed points $u_n^0$ given by the complex oscillations, $$u^0_n: S^1\to \IC,\,u_n^0(x)=\frac{1}{\sqrt{2\pi}}\exp(inx)\,\,\textrm{for every}\,\,n\in\IN.$$ \end{proposition}

From now on let $F_t$, $G_t=F_t\circ\phi^0_t$ and $\phi^0_t$, $\phi_t=\phi^0_t\circ\phi^G_t$ denote the corresponding functions and flows on the projective Hilbert space $\IP(\IH)=\IS(\IH)/S^1$. By generalizing Floer theory to the case of infinite-dimensional symplectic manifolds, in this paper we prove the following infinite-dimensional version of the Arnold conjecture. Recall that the Hofer norm of the time-periodic Hamiltonian $F_t$ is defined as $$|||F|||\,:=\,\int_0^1 (\max F_t - \min F_t)\;dt.$$

\begin{theorem} Assume that, after descending to the projectivization, the Hofer norm $|||F|||$ of the Hamiltonian defining the nonlinearity is smaller than $\pi/2$. Then for every fixed point $u^0_n$, $n\in\IN$ of the time-one map $\phi^0_1$ of the free Schr\"odinger equation there exists a fixed point $u^1_n$ of the time-one map $\phi_1$ of the given nonlinear Schr\"odinger equation of convolution type, and a Floer curve $\tilde{u}_n:\IR\times\IR\to\IP(L^2(S^1,\IC))$ which connects $u^0_n$ and $u^1_n$. Furthermore all fixed points $u^1_n$, $n\in\IN$ of $\phi_1:\IP(L^2(S^1,\IC))\to\IP(L^2(S^1,\IC))$ are different. \end{theorem}

We first explain the statement about the existence of a Floer curve connecting the fixed point $u^0_n$ of the free Schr\"odinger equation with a fixed point $u^1_n$ of the given Schr\"odinger equation of convolution type, which we view as a path $u^1_n:\IR\to\IP(\IH)$ with $u^1_n(t+1)=\phi^0_{-1}(u^1_n(t))$ and $\del_t u^1_n=X^G_t(u^1_n)$. With this we mean a smooth map $\tilde{u}=\tilde{u}_n:\IR\times\IR\to\IP(\IH)$ with $\tilde{u}(\cdot,t+1)=\phi^0_{-1}(\tilde{u}(\cdot,t))$ satisfying the Floer equation $$0\,=\,\CR\tilde{u}\;+\;\varphi(s)\cdot\nabla G_t(\tilde{u}),$$ where $\CR=\del_s+i\del_t$ denotes the standard Cauchy-Riemann operator and $\varphi$ is a smooth cut-off function with $\varphi(s)=0$ for $s\leq -1$ and $\varphi(s)=1$ for $s\geq 0$ and slope $0\leq\varphi'(s)\leq 2$. It connects $u^0_n$ and $u^1_n$ in the sense that there exist two sequences $(s_{\gamma}^{\pm})$ of real numbers, $s_{\gamma}^{\pm}\to\pm\infty$ with $\tilde{u}_n(s_{\gamma}^-,\cdot)\to u^0_n$,  $\tilde{u}_n(s_{\gamma}^+,\cdot)\to u^1_n$ as $\gamma\to\infty$ in the $C^0$-sense. Note that we cannot assume that $\tilde{u}_n(s,\cdot)\to u^1_n$ as $s\to\infty$, since we do \emph{not} want to assume that the nonlinearity is generic in the sense that all orbits are isolated. \\

Note that every fixed point $u^1_n\in\IP(L^2(S^1,\IC))$ corresponds to a weak solution of the nonlinear Schr\"odinger equation which, after taking the modulus, is periodic with respect to space and time. On the other hand, when the solution is smooth, then this leads to a strong solution. Indeed we actually prove the following strengthening of our main theorem.

\begin{proposition}\label{strong} The Floer curves $\tilde{u}_n$ as well as the corresponding fixed points $u^1_n$ sit in the weakly symplectic Frechet manifold  $\IP(C^{\infty}(S^1,\IC))\subset\IP(L^2(S^1,\IC))$. In particular, there are infinitely many different strong solutions $u:\IR\times\IR\to\IC$ of the equation $$i\del_t u \,=\, - \Delta u \;+\; \del_1  f(|u*\psi|^2,x,t)(u*\psi)*\psi$$ which are periodic in space and time with respect to the amplitude, $$|u(t+1,x)|\,=\,|u(t,x)|\,=\,|u(t,x+2\pi)|,$$ and are normalized in the sense that $$|u(t,\cdot)|_2^2\,:=\,\int_0^{2\pi} |u(t,x)|^2\;dx\,=\, 1.$$
\end{proposition} 

In fact, taking remark \ref{intrinsic} into account, we claim that there exists an intrinsic Hamiltonian Floer theory for weakly symplectic Frechet manifolds. Furthermore, defining $\tilde{\tilde{u}}(s,t,x):=\big(\phi^0_t(\tilde{u}(s,t))\big)(x)$, every Floer curve $\tilde{u}=\tilde{u}_n$ provides us with a strong solution $\tilde{\tilde{u}}:\IR\times\IR\times\IR\to\IC$ of a partial differential equation of perturbed Cauchy-Riemann-Schr\"odinger type,$$\del_s\tilde{\tilde{u}}\;+\;i\del_t\tilde{\tilde{u}} \,=\, - \Delta\tilde{\tilde{u}} \;+\; \varphi(s)\cdot \del_1  f(|\tilde{\tilde{u}}*\psi|^2,x,t)(\tilde{\tilde{u}}*\psi)*\psi,$$ satisfying periodicity and asymptotic conditions. 

\begin{remark} As shown in lemma \ref{approx}, the Hofer norm of any nonlinearity of convolution type is finite after passing to projective Hilbert space, so that the condition can always be fulfilled after multiplying $f$ by a suitably small positive number, or equivalently rescaling time and space coordinates. More precisely, in order to ensure that $|||F|||<\pi/2$, it suffices to have $|f(w,x,t)|<1/8$ for all $(w,x,t)\in [0,|\psi|_2^2]\times S^1\times [0,1]$, since from the latter it follows that the supremum norm of $F_t$ on $\IP(\IH)$ is bounded by $\pi/4$ for all $t$. The bound on the Hofer norm is required to ensure that the energy of the Floer curves is small enough to exclude bubbling-off of holomorphic spheres in infinite dimensions, which is needed to bound the derivatives and ultimately also to solve the small divisor problem. On the other hand, we would like to stress that no bound is required for the higher derivatives of $f$. 
\end{remark}

In contrast to the case of Floer theory in finite dimensions, we are for the first time faced with the famous small divisor problem that plays an important role in KAM theory, see \cite{EK} for the case of the nonlinear Schr\"odinger equation. Following the proof of proposition \ref{free-NLS}, the complex eigenvalues of the (linear) time-one flow map $\phi^0_1$ are given by the sequence $\exp(in^2)\in\IC$, $n\in\IN$. After restricting to a finite-dimensional subspace $\IC^{2k+1}\subset\IH$ and passing to the projectivization, it follows that all fixed points of $\phi^0_1$ are nondegenerate in the sense that one is not an eigenvalue of the linearized return map. On the other hand, after passing to the infinite-dimensional case, the latter is no longer true, as a subsequence of eigenvalues converges to $1$. In order to deal with the resulting lack of nondegeneracy of the time-one flow map of the free Schr\"odinger equation, we will use a deep result from the theory of diophantine approximations, proved using methods from analytic number theory. In essence, we have to use that the space period $2\pi$ cannot be approximated too well by rational multiples of the time period $1$.

\begin{remark} Before we turn to the proof of the main theorem, the following remarks are in order:
\begin{itemize} 
\item[i)] The problem of finding periodic solutions of Hamiltonian PDE has clearly attracted the interest of many researchers and many excellent papers exist on this topic: see the great book \cite{B} of M. Berti for an overview of the current state of the art. Apart from the non-perturbative result of P. Rabinowitz on the existence of time-periodic solutions of the nonlinear wave equation, see \cite{Ra} as well as \cite{BCN}, many deep perturbative results exist employing Kolmogorov-Arnold-Moser theory and the Lyapunov Center Theorem, see \cite{CW}, \cite{K2}, \cite{W}. On the other hand, when the nonlinearity is time-independent (autonomous), or has a simple time-dependence, periodic solutions can be found by prescribing the time dependence and by studying an elliptic PDE, see e.g. \cite{GP}. Just like the result of Rabinowitz, we stress that our result is \emph{non-perturbative}, and it is concerned with the general \emph{non-autonomous} case, i.e., we do not prescribe the time-dependence in any way.  
\item[ii)] In order to avoid the problem with small divisors, Rabinowitz considers the resonant case where the time period is a natural multiple of the space period. While the result in \cite{Ra} hence only holds for very special space periods, we consider the general \emph{non-resonant} case. In particular, we claim that our theorem continues to hold when the space period $2\pi$ is replaced by any other space period $L$ and the time period $1$ is replaced by any other time period $T$ as long as the quotient $\theta:=L^2/(2\pi\cdot T)$ is generic in the sense that it is a  diophantine real number, that is, there exist $r>0$ and $c>0$ such that $\inf_{p\in\IZ} |\theta - p/q|\geq c/q^r$ for all $q\in\IN$, see our follow-up paper \cite{FL}. While in the current paper we have decided to make concrete choices, both for the type of Hamiltonian PDE as well as the time and space periods, in \cite{FL} we study general Hamiltonian PDE with arbitrary time and space periods. We stress however that our result is definitely \emph{not} contained in \cite{FL} as a special case, since in the latter we work on linear Hilbert space and only prove the existence of a single forced periodic solution.    
\end{itemize}\end{remark}

Forgetting for the moment that we are working in the setting of infinite-dimensional symplectic manifolds, following Gromov's existence proof of symplectic fixed points in \cite{Gr} the idea would be to study moduli spaces of Floer curves $(\tilde{u},T)$, where $T>0$ is some non-negative real number and $\tilde{u}=\tilde{u}_{n,T}$ now denotes a smooth map $\tilde{u}:\IR\times\IR\to\IP(\IH)$ with $\tilde{u}(\cdot,t+1)=\phi^0_{-1}(\tilde{u}(\cdot,t))$, which satisfies the asymptotic condition $\tilde{u}(s,t)\to u^0_n$ as $s\to\pm\infty$ as well as the $T$-dependent perturbed Cauchy-Riemann equation $$\CR^T_G\tilde{u}\,=\,\CR\tilde{u} \;+\; \varphi_T(s)\cdot\nabla G_t(\tilde{u})\,=\, 0.$$ Here $\varphi_T:\IR\to \IR$ now denotes a smooth family of smooth cut-off functions with with $\varphi_0=0$ and $\varphi_T(s)=0$ for $s\leq -1$ and $s\geq 2T+1$, $\varphi_T(s)=1$ for $s\in [0,2T]$ and slope $0\leq\varphi_T'(s)\leq 2$ for $s<0$ and $-2\leq\varphi_T'(s)\leq 0$ for $s>0$ as long as $T\geq 1$. Assuming that, as in the case of finite-dimensional projective spaces, one could compactify the above moduli space by just adding broken holomorphic curves corresponding to the case that $T$ converges to $+\infty$, see \cite{DS} and the references therein, one would be able to show that for every $n\in\IN$ there exists a Floer curve connecting the fixed point $u^0_n$ of the free Schr\"odinger equation with a fixed point $u^1_n$ of the given Schr\"odinger equation of convolution type. In order to see that there are indeed infinitely many different fixed points $u^1_n$, we follow Floer's original proof of the degenerate Arnold conjecture for complex projective spaces. \\

On the other hand, it is quite apparent that the underlying theory of pseudo-holomorphic curves does not instantly carry over from finite to infinite dimensions. In particular, the non-compactness of the target manifold leads to the fact that Gromov's compactness theorem does not naturally generalize from finite-dimensional projective spaces to $\IP(\IH)$. For our proof we will make use of the fact that the Hamiltonian function $F_t$ (and hence $G_t$) on $\IP(\IH)$ defining the nonlinearity can be approximated uniformly by finite-dimensional Hamiltonian functions $F^k_t$ (and $G^k_t$) on $\CP^{2k}\subset\IP(\IH)$. Since the existence of Floer curves $\tilde{u}^k=\tilde{u}^k_{n,T_k}:\IR\times\IR\to\CP^{2k}\subset\IP(\IH)$ with $T_k\to\infty$ is guaranteed for every finite-dimensional Hamiltonian $F^k$, the idea is to show that a subsequence converges in the $C^{\infty}_{\loc}$-sense to a Floer curve $\tilde{u}=\tilde{u}_n:\IR\times\IR\to\IP(\IH)$ as in the main theorem. The fact that Gromov-Floer compactness still holds now relies on the following two key observations: First, the bubbling-off argument needed to uniformly bound the derivatives of the sequence $\tilde{u}^k$ still works. And secondly, although the target manifold is not compact and the time-one map is degenerate in the sense that a sequence of eigenvalues converges to $1$, $C^\infty$-convergence of the Floer curves can still be established as long as the infinite-dimensional nonlinearity is better approximated by finite-dimensional ones than that the space period $2\pi$ is approximated by rational multiples of the time period $1$. For the proof of the latter we use the aforementioned result from number theory.  \\

This paper is organized as follows: While in section $4$ we show how finite-dimensional symplectic Floer theory can be used to prove our main theorem in the special case of so-called finite-dimensional nonlinearities, for the general case of infinite-dimensional nonlinearities we prove that a suitable sequence of Floer curves in projective spaces of growing finite dimension has a subsequence converging to a Floer curve in projective Hilbert space as in the main theorem. While we introduce this sequence of finite-dimensional Floer curves in section $5$, we show in section $6$ that it has uniformly bounded derivatives, using a bubbling-off argument in projective Hilbert space. Together with a result from the theory of diophantine approximations, we will show that this will be sufficient to control $C^\infty$-convergence in infinite dimensions in section $7$ and finish the proof of our main theorem in section $8$. 

\section{Floer curves in complex projective spaces}

Before we turn to the general case, we first restrict ourselves to the case of finite-dimensional nonlinearities. Based on Floer's proof of the Arnold conjecture in finite dimensions we show

\begin{proposition}\label{finite-dim} Assume the nonlinearity is finite-dimensional in the sense that the support of the Fourier transform $\hat{\psi}:\IZ\to\IC$ of the smoothing kernel $\psi$ is finite, that is, $\textrm{supp}(\hat{\psi})\subset\{-k,\ldots,+k\}$ for some natural number $k$. Then the statement of the main theorem holds. \end{proposition}

Identifying the symplectic Hilbert space $\IH$ with $\ell^2(\IC)$ using the Fourier transform, it follows that $G$ just depends on its value after applying the projection $\pi_k:\IH\to\IC^{2k+1}$ onto the finite-dimensional symplectic subspace $\IC^{2k+1}=\{u\in\IH: \hat{u}(n)=0\,\,\textrm{for all}\,\,|n|>k\}$. In other words we have $G=G^k:=G\circ\pi_k$, so that at every point  the gradients $\nabla G_t$ and $X^G_t$ are vectors in $\IC^{2k+1}\subset\IH$. Note that, together with proposition \ref{free-NLS},  this implies that the flow $\phi_t$ on $\IP(\IH)$ restricts to a symplectic flow $\phi^k_t$ on $\CP^{2k}$. \\

With this the proof essentially relies on the following existence result of Floer curves in finite-dimensional complex projective spaces. From now on let us assume that the Hofer norm $|||F|||$ of $F_t$ on $\IP(\IH)$ is strictly smaller than $\pi/2$. Furthermore, let $\varphi_T:\IR\to \IR$, $T>0$ denote a smooth family of smooth cut-off functions as in section $3$, i.e.\ with $\varphi_0=0$ and $\varphi_T(s)=0$ for $s\leq -1$ and $s\geq 2T+1$, $\varphi_T(s)=1$ for $0\leq s\leq 2T$ and slope $0\leq\varphi_T'(s)\leq 2$ for $s<0$ and $-2\leq\varphi_T'(s)\leq 0$ for $s>0$ as long as $T\geq 1$. Furthermore the natural Riemannian norm on $\CP^{2k}$ is denoted by $|\cdot|$.\\

For the following statement we need to assume that the Hamiltonian $F_t$ is regular in the sense that the transversality holds for a certain nonlinear Cauchy-Riemann operator. Note that this can be achieved after adding an arbitrarily small generic time-dependent perturbation to the original time-dependent Hamiltonian $F_t$.   

\begin{proposition}\label{curves} Let $k\in\IN$ and $n\leq k$. Possibly after adding a small generic time-periodic perturbation to $F_t$, there exists a connected moduli space $\IM^{k,n}$ of tuples $(\tilde{u},T)$ consisting of non-negative real number $T$ and a smooth map $\tilde{u}=\tilde{u}_n=\tilde{u}^k_{n,T}:\IR\times\IR\to\CP^{2k}$, called \emph{Floer curve}, satisfying the periodicity condition $\tilde{u}(\cdot,t+1)=\phi^0_{-1}(\tilde{u}(\cdot,t))$, the asymptotic condition $\tilde{u}_n(s,\cdot)\to u^0_n$ as $s\to\pm\infty$, and the perturbed Cauchy-Riemann equation $$0=\CR^T_G\tilde{u}=\CR\tilde{u} + \varphi_T(s)\cdot \nabla G^k_t(\tilde{u}),$$ such that the canonical projection $\IM^{k,n}\to\IR^+\cup\{0\}$, $(\tilde{u},T)\mapsto T$ is surjective. Furthermore, for the resulting families of maps $\tilde{u}$ we have that the energy $E(\tilde{u})$ defined by $$E(\tilde{u}):=\int_{-\infty}^{+\infty}\int_0^1 \frac{1}{2}\big(|\del_s\tilde{u}|^2 + |\del_t\tilde{u} - \varphi_T(s)\cdot X^{G,k}(\tilde{u})|^2\big)\;dt\;ds$$ is bounded by $2 |||G^k|||<\pi$. \end{proposition} 

\begin{proof} While for the statement we use the formalism of Floer homology for general symplectomorphisms from \cite{DS}, everything can be translated into the more established language of Floer homology for Hamiltonian symplectomorphisms used in the standard reference \cite{S}: Since the time-one map $\phi^0_t$ of the free Hamiltonian $H^0$ is Hamiltonian when restricted to the \emph{finite}-dimensional manifold $\CP^{2k}$, there is a one-to-one correspondence between maps $\tilde{u}=\tilde{u}^k_{n,T}:\IR\times\IR\to\CP^{2k}$ as in the statement and maps $\tilde{\tilde{u}}=\tilde{\tilde{u}}^k_{n,T}:\RS\to\CP^{2k}$ satisfying the asymptotic condition $\tilde{\tilde{u}}_n(s,\cdot)\to u^0_n$ as $s\to\pm\infty$, and the perturbed Cauchy-Riemann equation $$0=\CR^T_H(\tilde{\tilde{u}})=\CR{\tilde{\tilde{u}}} + \nabla H_{s,t}({\tilde{\tilde{u}}})\,\,\textrm{with}\,\,H_{s,t}=H^0 + \varphi_T(s)\cdot F_t,$$ given by $\tilde{\tilde{u}}(s,t)=\phi^0_t(\tilde{u}(s,t))$. Here we also use that the flow $\phi^0_t$ of the free Hamiltonian $H^0$ preserves the complex structure $i$ on $\CP^{2k}$.\\

For $T=0$ the constant curve $\tilde{\tilde{u}}^k_{n,0}(s,t)=u^0_n$ staying over the critical point $u^0_n\in \CP^{2k}$ of $H^0$ is the unique solution. Possibly after adding an arbitrarily small generic time-dependent perturbation, we may assume that transversality for the Cauchy-Riemann operator $\CR_H$ given by $\CR_H(\tilde{\tilde{u}},T)=\CR^T_H\tilde{\tilde{u}}$ holds. Then it follows that there is a connected moduli space $\IM^{k,n}$ of tuples $(\tilde{\tilde{u}},T)$ containing $(u^0_n,0)$ which forms a one-dimensional manifold. On the other hand, the existence of a Floer curve $\tilde{\tilde{u}}^k_{n,T}$ for all $T>0$ in $\IM^{k,n}$ follows from the Gromov-Floer compactness result, as we can exclude bubbling-off of holomorphic spheres as well as breaking-off of cylinders for finite $T$. Concerning the first claim, note that bubbling-off of holomorphic spheres is excluded due to fact that the energy $E(\tilde{\tilde{u}})=E(\tilde{u})$ is bounded from above by twice the Hofer norm of the Hamiltonian $G^k$, and the Hofer norm of $G^k$ is smaller than 1/2 the minimal energy of a holomorphic sphere in $\CP^{2k}$, where we also use that the curve $\tilde{\tilde{u}}^k_{n,T}$ is homotopic to the constant curve $\tilde{\tilde{u}}^k_{n,0}=u^0_n$. Concerning the second claim, note that, due to the bounded support of $\varphi_T$ for finite $T$, the Floer curve would have to break at another fixed point $u^0_m$, $m\neq n$ of the free flow. Since all critical points $u^0_n$ of the canonical Morse function $H^0$ on $\CP^{2k}$ have even Morse index and hence even Conley-Zehnder index, the latter is excluded by index reasons. \end{proof}

Combined with our novel infinite-dimensional Gromov-Floer compactness result in the presence of small divisors, the above result will turn out is sufficient to prove the existence of a single time-periodic solution. In particular, we can directly show that the main theorem holds for true in the case of finite-dimensional nonlinearities. 

\begin{proof}\emph{(of proposition \ref{finite-dim})} Note that the original Hamiltonian $F_t$ can be approximated by a family of perturbed $t$-dependent Hamitonians $F^{\nu}_t:\CP^{2k}\to\IR$ with $F^{\nu}_t\to F_t$ uniformly with all derivatives as $\nu\to 0$. As $T$ tends to $+\infty$ and $\nu$ converges to zero, it follows from the uniform energy bound together with elliptic bootstrapping that the sequence $\tilde{u}_T=\tilde{u}^k_{n,T}$ of Floer curves converges in the $C^{\infty}_{\loc}$-sense to a smooth map $\tilde{u}:\IR\times\IR\to\CP^{2k}$ with $\tilde{u}(\cdot,t+1)=\phi^0_{-1}(\tilde{u}(\cdot,t))$. The map $\tilde{u}$ satisfies the Floer equation $0=\CR\tilde{u}+\varphi(s)\cdot\nabla G_t(\tilde{u}),$ where $\varphi$ is now a smooth cut-off function with $\varphi(s)=0$ for $s\leq -1$ and $\varphi(s)=1$ for $s\geq 0$. It connects the fixed point $u^0_n$ of the free Schr\"odinger equation with a fixed point $u^1_n$ of $\phi_1$ in the sense that there exist sequences $(s_{\gamma}^{\pm})$ of positive and negative real numbers, $s_{\gamma}^{\pm}\to\pm\infty$ with $\tilde{u}(s_{\gamma}^-,\cdot)\to u^0_n$ and $\tilde{u}(s_{\gamma}^+,\cdot)\to u^1_n$ as $\gamma\to\infty$. In order to see the latter, note that by the bound for the energy $E(\tilde{u})$ from proposition \ref{curves}, it follows that for every $\gamma\in\IN$ there exists $\gamma\leq |s_{\gamma}^{\pm}|\leq 2\gamma$ such that $$\int_0^1 |\del_t\tilde{u}(s_{\gamma}^{\pm},t) - \varphi(s_{\gamma}^{\pm})X^{G,k}_t(\tilde{u}(s_{\gamma}^{\pm},t))|^2\;dt\,<\,\frac{\pi}{\gamma}.$$ By compactness of $\CP^{2k}$ we know, possibly after passing to a subsequence, that the sequence $\tilde{u}(s_{\gamma}^{\pm},0)$ converges to a fixed point of $\phi^0_1$ or $\phi_1$, respectively. In order to see that $\tilde{u}(s_{\gamma}^-,\cdot)$ indeed converges to the fixed point $u^0_m$ of $\phi^0_1$ with $m=n$, recall that breaking-off of cylinders for the free Hamiltonian $H^0$ can be excluded by index reasons. \\

In order to see that the resulting fixed points $u^1_m$, $u^1_n$ for $H_t=H^0+F_t$ are different if $m\neq n$ as well, it suffices to refer to Floer's proof in \cite{Fl} of the degenerate Arnold conjecture for $\CP^{2k}$ using the cup action on Floer cohomology. Since Floer cohomology can only be defined for Hamiltonians with nondegenerate one-periodic orbits, we consider again the sequence of perturbed Hamiltonians $H^{\nu}_t=H^0+F^{\nu}_t$ with $H^{\nu}_t\to H_t$ from before. Then the Floer cohomology groups $\HF^*=\HF^*(H^{\nu}_t)$ are independent of the chosen nondegenerate Hamiltonian and, on the chain level, generated by lifts of the one-periodic orbits of $H^{\nu}_t$ to the universal cover $\widetilde{\LL}(\CP^{2k})$ of the loop space $\LL(\CP^{2k})$ of $\CP^{2k}$. Working with the universal cover is required, since the Hamiltonian action functional as well as the Conley-Zehnder index are only well-defined after choosing a filling disk for each orbit. In particular, the Floer cohomology is cyclic in the sense that $\HF^*=\HF^{*+2N}$ where $N=2k$ is the minimal Chern number of $\CP^{2k}$. Furthermore there is a nontrivial cup action of the singular cohomology $H^*(\CP^{2k})$ on Floer homology $\HF^*$ which in the case of the canonical generator $\PD[\CP^{2k-1}]\in H^2(\CP^{2k})$ defines a linear isomorphism between $\HF^*$ and $\HF^{*+2}$. For this one uses that, when $F^{\nu}_t=0$, $H^{\nu}_t=H^0$ is the canonical Morse function on $\CP^{2k}$ which has the property that every Morse gradient flow line between $u^0_n$ and $u^0_{n+1}$ intersects every representative of $[\CP^{2k-1}]$ precisely once; a similar statement holds for every Morse gradient flow line between $u^0_{2k}$ and $u^0_1$. Denoting for each $n\leq k$ by $\tilde{u}^{1,\nu}_n$ the lifts to $\widetilde{\LL}(\CP^{2k})$ of fixed points of $H^{\nu}_t$ which converge to $u^1_n$ as $\nu\to\infty$, it follows from the nontriviality of the cup action (or equivalently by established compactness and gluing arguments) that there exist connecting Floer curves $\tilde{\tilde{u}}^{\nu}_{n,n+1}:\RS\to\CP^{2k}$, $1\leq n\leq 2k-1$ and $\tilde{\tilde{u}}^{\nu}_{2k,1}:\RS\to\CP^{2k}$ satisfying the perturbed Cauchy-Riemann equation $$0= \CR\tilde{\tilde{u}}^{\nu}_{m,n}+\nabla H^{\nu}_t(\tilde{\tilde{u}}^{\nu}_{m,n}),$$ the intersection property $\tilde{\tilde{u}}^{\nu}_{m,n}(0,0)\in [\CP^{2k-1}]$, as well as the asymptotic condition $\tilde{\tilde{u}}^{\nu}_{m,n}(s,\cdot)\to\tilde{u}^{1,\nu}_m,\tilde{u}^{1,\nu}_n$ as $s\to\pm\infty$. Here $[\CP^{2k-1}]$ denotes any chosen fixed pseudocycle representing this class and $\tilde{u}^{1,\nu}_m,\tilde{u}^{1,\nu}_n$ are lifts of $u^{1,\nu}_m$, $u^{1,\nu}_n$ such that, after gluing in the corresponding filling disks, the Floer curve is contractible. Since for their Conley-Zehnder indices we have $\CZ(\tilde{u}^{1,\nu}_n)-\CZ(\tilde{u}^{1,\nu}_m)=2$ and the minimal Chern number is $2k$, it follows that the underlying fixed points $u^{1,\nu}_m$ and $u^{1,\nu}_n$ need to be different. By letting $\nu$ converge to zero, and hence $H^{\nu}_t$ converge to $H_t$, it follows that $u^1_m\neq u^1_n$, since there still must exist a nontrivial Floer curve which intersects the representative of $[\CP^{2k-1}]$. \\

On the other hand, for $n>k$ it follows that the fixed points $u^0_n$ of the free Schr\"odinger equation are also fixed points of $\phi_1$, where the corresponding Floer curve $\tilde{u}=\tilde{u}_n$ is just the constant curve $\tilde{u}^k_{n,0}(s,t)=u^0_n=u^1_n$, $(s,t)\in\IR\times\IR$. \end{proof}

In the case of finite-dimensional nonlinearities we see that the main theorem can be proven by studying Floer curves in finite-dimensional complex projective spaces. In preparation for the general statement, we first show that everything is independent of the chosen ambient finite-dimensional projective space. For this we want to restrict ourselves to the case where the Hamiltonian $H_t=H^0+F_t$ is regular in the sense that the connected moduli space $\IM^{k,n}$ of tuples $(\tilde{\tilde{u}},T)$ containing $(u^0_n,0)$ from the proof of proposition \ref{curves} is a one-dimensional manifold. For the following statement it is actually crucial to restrict ourselves to this connected component of the moduli space and ignore any other components which are anyways irrelevant for the existence proof of Floer curves for all $T>0$. 

\begin{proposition}\label{Liouville} Assume that $\supp(\hat{\psi})\subset\{-\ell,\ldots,+\ell\}$. Then for every tuple $(\tilde{u},T)$ from $\IM^{k,n}$ we have $\tilde{u}(s,t)\in\CP^{2\ell}\subset\CP^{2k}$ for all $(s,t)\in\IR\times\IR$. \end{proposition} 

\begin{proof} It is known that $\CP^{2k}\backslash D_{k,\ell}$, $D_{k,\ell}:=i_{k,\ell}(\CP^{2k-2\ell-1})$ carries the structure of a $\IC^{2k-2\ell}$-bundle over $\CP^{2\ell}$, where the embedding of $\CP^{2k-2\ell-1}$ into $\CP^{2k}$ consists of all points $[z_{-k}:\ldots:z_{+k}]$ with $z_{-\ell}=\ldots=z_{+\ell}=0$. If one would add the removed embedding of $\CP^{2k-2\ell-1}$ back in, each fibre $\IC^{2k-2\ell}$ would get compactified to $\CP^{2k-2\ell}$ using the Hopf map. We want to emphasize however that the total space $E^{k,\ell}$ of the resulting $\CP^{2k-2\ell}$-bundle over $\CP^{2\ell}$ is \emph{not} diffeomorphic to $\CP^{2k}$; in order to obtain $\CP^{2k}$ we rather would have to assume that all fibres intersect in the removed divisor $D_{k,\ell}$. \\

For the proof of the statement it suffices to prove the following\\

\noindent\emph{Claim: For all $(\tilde{u},T)\in\IM^{k,n}$ for which $\tilde{u}=\tilde{u}^k_{n,T}$ has image in $\CP^{2k}\backslash D_{k,\ell}$ it holds that $\tilde{u}$ has image in $\CP^{2\ell}$.} \\

The statement of the proposition then immediately follows using that $\tilde{u}^k_{n,0}=u^0_n\in\CP^{2n}\subset\CP^{2\ell}$ for $T=0$ and the fact that the connected moduli space $\IM^{k,n}$ is a connected one-dimensional manifold which is continuous with respect to the underlying $H^{1,p}$-norm and hence also the weaker $C^0$-norm. Indeed, considering the subset of tuples $(\tilde{u},T)$ for which $\tilde{u}$ has image in $\CP^{2\ell}$, it follows from continuity that it is closed and, after additionally employing the claim, that it is open as well. \\

For the proof of the claim in the following we now assume that the Floer curve $\tilde{u}$ has image in the total space of the bundle $E^{k,\ell}\to\CP^{2\ell}$. Then we can write the Floer curve as a pair of maps, $$\tilde{u}=(\tilde{u}^{\ell}_{\perp},\tilde{u}^{\ell}):\IR\times\IR\to \IC^{2k-2\ell}\times \CP^{2\ell},$$ where $\tilde{u}^{\ell}$ is the projection onto $\CP^{2\ell}$ and $\tilde{u}^{\ell}_{\perp}$ remembers the normal component. For the latter observe that we can view $\tilde{u}$, $\tilde{u}^{\ell}$ as sections $\check{u}$, $\check{u}^{\ell}$ in the bundles $\IR\times M_{\phi_1^0}\to\RS$ for $M=E^{k,\ell}$, $\CP^{2\ell}$, respectively, where
$$
M_{\phi_1^0}:=\frac{\IR\times M}{(t+1,u)\sim(t,\phi^0_1(u))}
$$
and $\check{u}(s,t)=(s,t,\tilde{u}(s,t))$. Then $\tilde{u}^{\ell}_{\perp}$ can be viewed as a section in the pull-back of the $\IC^{2k-2\ell}$-bundle $\IR\times E^{k,\ell}_{\phi_1^0}\to\IR\times\CP^{2\ell}_{\phi_1^0}$ under the section $\check{u}^{\ell}:\RS\to\IR\times\CP^{2\ell}_{\phi_1^0}$; since $\tilde{u}^{\ell}=u^0_n$ for $T=0$, any section $\tilde{u}^{\ell}_{\perp}$ can be viewed as a smooth map $\IR\times\IR\to\IC^{2k-2\ell}$ satisfying the periodicity condition $\tilde{u}^{\ell}_{\perp}(\cdot,t+1)=\phi^0_{-1}(\tilde{u}(\cdot,t))$. \\

Now the important observation is that, since $G=G^{\ell}=G\circ\pi_{\ell}$, the normal component $\nabla^{\ell}_{\perp} G_t$ of $\nabla G_t$ vanishes. This however implies that the normal component $\tilde{u}^{\ell}_{\perp}$ is truely holomorphic, that is, solves the unperturbed Cauchy-Riemann equation $\CR \tilde{u}^{\ell}_{\perp}=0$. Since $\tilde{u}^{\ell}_{\perp}(s,t)\to 0$ for $s\to\pm\infty$ since $u^0_n\in\CP^{2\ell}$, we can employ Liouville's theorem to show that we have $\tilde{u}^{\ell}_{\perp}=0$, that is, $\tilde{u}=\tilde{u}^{\ell}$. Note that, instead of referring to Liouville's theorem, the result can be viewed as a consequence of the minimal surface property of pseudo-holomorphic curves. \end{proof}

\begin{remark}\label{Liouville-2} The following observations are immediate:
\begin{itemize}
\item[i)] By the same arguments it follows that, even if we first allowed the Floer curve $\tilde{u}$ to live in the infinite-dimensional manifold $\IP(\IH)$, the finite-dimensionality of the nonlinearity ensures that it actually lives in the finite-dimensional submanifold $\CP^{2\ell}$. 
\item[ii)] Along the same lines using Liouville's theorem or the minimal surface property, it is immediate to see that, in the case of $n>\ell$, the Floer curve is constant and the fixed point $u^0_n\in\CP^{2n}$ of the free Schr\"odinger equation thus agrees with the corresponding fixed point $u^1_n$ of the nonlinear Schr\"odinger equation with convolution term. 
\end{itemize}
\end{remark}

\section{From finite to infinite dimensions}

In this section we start with the proof of the main theorem. 

\begin{theorem} Assume that, after descending to the projectivization, the Hofer norm $|||F|||$ of the Hamiltonian defining the nonlinearity is smaller than $\pi/2$. Then for every fixed point $u^0_n$, $n\in\IN$ of the time-one map $\phi^0_1$ of the free Schr\"odinger equation there exists a fixed point $u^1_n$ of the time-one map $\phi_1$ of the given nonlinear Schr\"odinger equation of convolution type, and a Floer curve $\tilde{u}_n:\IR\times\IR\to\IP(L^2(S^1,\IC))$ which connects $u^0_n$ and $u^1_n$. Furthermore all fixed points $u^1_n$, $n\in\IN$ of $\phi_1:\IP(L^2(S^1,\IC))\to\IP(L^2(S^1,\IC))$ are different. \end{theorem}

Until further notice let us fix the natural number $n\in\IN$. Recall that a Floer curve connecting the fixed point $u^0_n$ of the free Schr\"odinger equation with a fixed point $u^1_n$ of the given Schr\"odinger equation of convolution type is a smooth map $\tilde{u}=\tilde{u}_n:\IR\times\IR\to\IP(\IH)$ with $\tilde{u}(\cdot,t+1)=\phi^0_{-1}(\tilde{u}(\cdot,t))$ satisfying the Floer equation $$0\,=\,\CR\tilde{u}\;+\;\varphi(s)\cdot\nabla G_t(\tilde{u}),$$ where $\CR=\del_s+i\del_t$ denotes the standard Cauchy-Riemann operator and $\varphi$ is a smooth cut-off function with $\varphi(s)=0$ for $s\leq -1$ and $\varphi(s)=1$ for $s\geq 0$. It connects $u^0_n$ and $u^1_n$ in the sense that there exist two sequences $(s_{\gamma}^{\pm})$ of real numbers, $s_{\gamma}^{\pm}\to\pm\infty$ with $\tilde{u}_n(s_{\gamma}^-,\cdot)\to u^0_n$, $\tilde{u}_n(s_{\gamma}^+,\cdot)\to u^1_n$ as $\gamma\to\infty$ in the $C^0$-sense; the latter weaker asymptotic condition is a consequence of the fact that we do \emph{not} want to assume that the nonlinearity is generic in the sense that all orbits are isolated. \\

As mentioned in section $3$, as a starting point we make use of the fact, proven in lemma \ref{approx} and lemma \ref{exp}, that the infinite-dimensional nonlinearity $G_t=F_t\circ\phi^0_t$ can be uniformly approximated by the finite-dimensional Hamiltonians $G^k_t=F^k_t\circ\phi^0_t$, together with the fact that for every finite-dimensional Hamiltonian $G^k$ Floer curves are known to exist. In what follows we want to assume without loss of generality that each of finite-dimensional Hamiltonians $F^k_t=F_t\circ\pi_k$ is regular in the sense that the resulting moduli space $\IM^{k,n}$ is a one-dimensional manifold. If this is not the case, then we redefine $F^k_t:=F^{k,\nu}_t=F_t\circ\pi_k+\nu^k_t$ where the time-dependent perturbation $\nu^k_t$ is chosen to decay exponentially fast with $k\in\IN$ in order to ensure that the statement of lemma \ref{approx} still holds.\\

Choose $T_k:=k$ for all $k\in\IN$. By proposition \ref{curves}, for every $n\leq k$ there exists a Floer curve $\tilde{u}^k=\tilde{u}^k_n=\tilde{u}^k_{n,T_k}:\IR\times\IR\to\CP^{2k}$ for the finite-dimensional time-dependent Hamiltonian $G^k$ satisfying the periodicity condition $\tilde{u}^k(\cdot,t+1)=\phi^0_{-1}(\tilde{u}^k(\cdot,t))$, the asymptotic condition $\tilde{u}^k_n(s,\cdot)\to u^0_n$ as $s\to\pm\infty$, and the perturbed Cauchy-Riemann equation $$0=\CR\tilde{u}^k + \varphi_{T_k}(s)\cdot \nabla G^k_t(\tilde{u}^k).$$ Furthermore, we have that the energy $E(\tilde{u}^k)$ defined by $$E(\tilde{u}^k):=\int_{-\infty}^{+\infty}\int_0^1 \frac{1}{2}\big(|\del_s\tilde{u}^k|^2 + |\del_t\tilde{u}^k - \varphi_{T_k}(s)\cdot X^{G,k}(\tilde{u}^k)|^2\big)\;dt\;ds$$ is bounded by $2 |||G^k|||$, which is strictly less than the minimal energy of a holomorphic sphere in $\IP(\IH)$ for $k$ sufficiently large. \\

We are going to show that, possibly after passing to a subsequence, the sequence of Floer curves $\tilde{u}^k=\tilde{u}^k_n:\IR\times\IR\to\CP^{2k}\subset\IP(\IH)$ will converge in the $C^{\infty}_{\loc}$-sense to a Floer curve $\tilde{u}=\tilde{u}_n:\IR\times\IR\to\IP(\IH)$ for the infinite-dimensional time-dependent Hamiltonian $G$ as in the main theorem. Our proof will rely on bubbling-off analysis in infinite dimensions as well as a result from the theory of diophantine approximations in order to deal with the small divisor problem. \\

\section{Bounded derivatives using bubbling-off analysis}

As a first step we prove the following   

\begin{proposition}\label{no-bubbling} For all $\alpha\in\IN$ the $C^{\alpha}$-norm of the maps $(\tilde{u},T)\in\IM^{k,n}$, $k\in\IN$ is uniformly bounded. In other words, we have $$\sup_k \|\tilde{u}\|_{C^{\alpha}}\,<\,\infty\,\,\textrm{for all}\,\,\alpha\in\IN, $$ where the supremum is taken over all Floer curves in $\IM^{k,n}$ and all dimensions $k\in\IN$. \end{proposition}

For the proof we use an infinite-dimensional version of the classical bubbling-off argument from (\cite{MDSa}, chapter 4) together with elliptic regularity from (\cite{MDSa}, appendix B). Although the proof follows the lines of its finite-dimensional counterpart, the crucial observation is that the bubbling-off argument can indeed be adapted to the infinite-dimensional setting. 

\begin{lemma}\label{C^1-bound} The first derivatives $\del_s\tilde{u}(s,t)$, $\del_t\tilde{u}(s,t)$ can be uniformly bounded, that is, $$\sup_k \|T\tilde{u}\|_{\infty}\,<\,\infty.$$ \end{lemma}

\begin{proof} In order to show that the supremum norm of $\del_s\tilde{u}$ is bounded, we are now essentially going to use that the energy $E(\tilde{u})$ of $\tilde{u}$ is strictly smaller than the minimal energy of a holomorphic sphere in $\CP^{2k}\subset\IP(\IH)$, at least as long as $k$ is sufficiently large. Under the assumption that the gradient explodes, in contrast to the classical proof from finite dimensions, we are not going to prove the existence of a holomorphic sphere in order to derive a contradiction. In order to circumvent the corresponding generalization of the underlying Gromov compactness statement to infinite dimensions, we instead show that, when the gradient is large enough, a bubble can be formed by adding in a small disk of diameter smaller than the injectivity radius. The latter will be sufficient to derive a contradiction. Note that, since $\CR\tilde{u}+\varphi_T(s)\nabla G^k_t(\tilde{u})=0$ and $\nabla G^k\to\nabla G$ as $k\to\infty$ due to lemma \ref{approx}, a bound for the supremum norm of $\del_s\tilde{u}$ implies that also the supremum norm of $\del_t\tilde{u}$ is bounded. \\

To the contrary, assume that, possibly after passing to a subsequence, there exists $(\tilde{u}^k,T_k)\in\IM^{k,n}$, $\tilde{u}^k=\tilde{u}^k_{n,T_k}$ such that we have that $C_k:=\max\{|\del_s\tilde{u}^k(z)|:z\in\IR\times\IR\}\to\infty$ as $k\to\infty$, where we may additionally assume that $T_k\to T_{\infty}\in\IR^+\cup\{\infty\}$ as $k\to\infty$. For the start choose for every $k\in\IN$ a point $z^k_0$ such that $|\del_s\tilde{u}^k(z^k_0)|=C_k$. In the proof of this lemma the norm refers to the standard Riemannian metric on $\CP^{2k}$; note that establishing a bound for the Riemannian metric on $\CP^{2k}$ establishes a bound in terms of the Euclidean metric on $\IC^{2k}$ after applying the coordinate chart $\varphi_n$. Note that the maximum exists, due to the asymptotic condition. As in the finite-dimensional bubbling-off proof we define $\tilde{v}^k: B^2_{\sqrt{C_k}}(0)\to\CP^{2k}$ by $\tilde{v}^k(z):=\tilde{u}^k(z/C_k+z^k_0)$, so that $|\del_s\tilde{v}^k(0)|=1$ and $|\del_s\tilde{v}^k(z)|\leq 1$ for all $z\in B^2_{\sqrt{C_k}}(0)$.  For each $r\in[0,\sqrt{C_k}]\subset\IR^+$ define $\gamma^k_r: S^1\to\CP^{2k}$ by $\gamma^k_r(\theta):=\tilde{v}^k(re^{i\theta})$. Denote by $\ell$ the map which assigns to each loop $\gamma:S^1\to\CP^{2k}$ its length with respect to the canonical Riemannian metric and $E_{\omega}(v):=\int v^*\omega$ the symplectic area of a disk map $v: B^2_r(0)\to\CP^{2k}$. With this we can formulate the following  \\

\noindent\emph{Claim: For every $k$ there exists some $\rho_k\in[\sqrt{C_k}/2,\sqrt{C_k}]$ such that $\ell(\gamma^k_{\rho_k})\to 0$. Furthermore, for sufficiently large $k$, we have for the symplectic area of the restricted map $\tilde{v}^k_{\rho_k}=\tilde{v}^k:B^2_{\rho_k}(0)\to\CP^{2k}$ that $E_{\omega}(\tilde{v}^k_{\rho_k})\leq 2|||G_k|||<\pi$ and the a priori estimate $|\del_s\tilde{v}^k(0)|^2 < E_{\omega}(\tilde{v}^k_{\epsilon})/(\epsilon^2\pi)$ holds for $\eps>0$ sufficiently small.}\\

Setting $\tilde{w}^k(r,\theta)=\tilde{v}^k(r e^{i\theta})$ and using the finiteness of the $C^1$-norm of $G_k$, one shows that $$E_{\omega}(\tilde{v}^k)\,-\,\int_{B^2_{1/\sqrt{C_k}}(z_0^k)} \frac{1}{2}\big(|\del_s\tilde{u}^k|^2+|\del_t\tilde{u}^k-\varphi_{T_k}(s)X^{G,k}_t(\tilde{u}^k)|^2\big)\;ds\;dt\,\to\, 0$$ and hence $$\int_0^{\sqrt{C_k}}\int_0^{2\pi} |\del_{\theta}\tilde{w}^k|^2\;rd\theta\;dr\,-\,\int_{B^2_{1/\sqrt{C_k}}(z_0^k)}\frac{1}{2}\big(|\del_s\tilde{u}^k|^2+|\del_t\tilde{u}^k-\varphi_{T_k}(s) X^{G,k}_t(\tilde{u}^k)|^2\big)\;ds\;dt$$ converges to $0$ as $C_k\to\infty$. Together with $E(\tilde{u}^k)\leq 2|||G_k|||$ and Cauchy-Schwarz, this implies that $$(2\pi)^{-1}\cdot \int_{\sqrt{C_k}/2}^{\sqrt{C_k}}\Big(\int_0^{2\pi}|\del_{\theta}\tilde{w}^k|d\theta\Big)^2\;rdr\,\leq\,\int_{\sqrt{C_k}/2}^{\sqrt{C_k}}\int_0^{2\pi} |\del_{\theta}\tilde{w}^k|^2\;d\theta\;rdr<\pi$$ for $k$ sufficiently large. In particular, by setting $$\ell^k_{\min}:=\min\{\ell(\gamma^k_r):\;r\in [\sqrt{C_k}/2,\sqrt{C_k}]\},$$ it follows that $\ell^k_{\min}\leq \sqrt{2\pi^2/(\sqrt{C_k}/2))}\to 0$ as $k\to\infty$; in other words, for every $\eps>0$ there exists $k_0\geq n$ such that $\ell^k_{\min}<\eps$ if $k\geq k_0$. \\

For the result on the symplectic area it suffices to observe that $E_{\omega}(\tilde{v}^k)< 2|||G_k|||<\pi$ for $k$ sufficiently large. On the other hand, for the a priori estimate, we start by observing that $$\CR\tilde{v}^k\,=\,-C_k^{-1}\varphi_{T_k}(s)\nabla G^k_t(\tilde{v}^k).$$ Using $\Delta=\del\circ\CR$ with $\del=\del_s-i\del_t$ it follows that $$\Delta \del_s \tilde{v}^k = (\del\circ\del_s)(-C_k^{-1}\varphi_{T_k}(s)\nabla G^k_t(\tilde{v}^k))\,\to\, 0\,\,\textrm{as}\,\,k\to\infty.$$ On the other hand, setting $v:=\del_s\tilde{v}^k$, it follows from the divergence theorem that for every $r>0$ sufficiently small $$\frac{\del}{\del r} \Big(\int_0^{2\pi} v(r e^{i\theta})\;d\theta\Big)\,=\,\frac{1}{r}\int_{B_{r}(0)} \Delta v(z)\;dz\,\to\, 0\,\,\textrm{as}\,\,k\to\infty.$$ But this implies that $$v(0) - \frac{1}{\eps^2\pi}\int_{B_{\eps}(0)} v(z)\;dz \,\to\, 0\,\,\textrm{as}\,\,k\to\infty,$$ which together with $$\frac{1}{\eps^2\pi}\Big|\int_{B_{\eps}(0)} v(z)\;dz\Big|\,\leq\,\frac{1}{\eps\pi^{1/2}}\Big(\int_{B_{\eps}(0)} |v(z)|^2\;dz\Big)^{1/2}\,\leq\,\frac{1}{\eps\pi^{1/2}}\|\del_s \tilde{v}\|_2$$ implies the claim. \\

In order to finish the proof of the lemma we observe that, due to the fact that the length of $\gamma^k_{\rho_k}$ converges to zero, for $k$ sufficiently large there exists a filling disk $\tilde{\gamma}^k_{\rho_k}: B^2_{1/\rho_k}(0)\to\CP^{2k}$ with $E_{\omega}(\tilde{\gamma}^k_{\rho_k})\to 0$ as $k\to\infty$: Indeed it is shown in remark 4.4.2 in \cite{MDSa} that, when the length of $\gamma: S^1\to\CP^{2k}$ is smaller than half of the injectivity radius, the map $\gamma$ has a canonical local extension to a map $\tilde{\gamma}$ from the disk defined using the exponential map; further it is shown there that there exist $\ell_{\max}>0$ and $c>0$ such that $E_{\omega}(\tilde{\gamma})\leq c\ell(\gamma)^2$ if $\ell(\gamma)\leq\ell_{\max}$. Even more important, due to the symmetries of the canonical Riemannian metric on $\CP^{2k}$ (and the fact that the embedding of $\CP^{2\ell}$ into $\CP^{2k}$ respects the metric for $\ell\leq k$), it follows that the constants $\ell_{\max}>0$ and $c>0$ are actually independent of the complex dimension $2k$.\\ 

Since $\tilde{v}^k_{\rho_k}$ and $\tilde{\gamma}^k_{\rho_k}$ match on their boundaries, it follows that $E_{\omega}(\tilde{v}^k_{\rho_k})+E_{\omega}(\tilde{\gamma}^k_{\rho_k})=m\pi$ for some $m\in\IN$. But since $E_{\omega}(\tilde{v}^k_{\rho_k})+E_{\omega}(\tilde{\gamma}^k_{\rho_k})<\pi$ for $k$ sufficiently large, it follows that $m=0$, in particular, $E_{\omega}(\tilde{v}^k_{\rho_k})\to 0$ as $k\to\infty$. Applying now the a priori estimate $|\del_s\tilde{v}^k(0)|^2 < E_{\omega}(\tilde{v}^k_{\epsilon})/(\epsilon^2\pi)$ for $k$ sufficiently large, it follows that $\del_s\tilde{v}^k(0)\to 0$ - in contradiction to $|\del_s\tilde{v}^k(0)|=1$. \end{proof}

\begin{proof}\emph{(of proposition \ref{no-bubbling})} With the help of the above lemma, we can now give the proof of proposition \ref{Floer-smooth}. In order to keep the setup sufficiently simple, we will assume that $\tilde{u}^k=\tilde{u}^k_{n,T_k}$ has image in the image of the corresponding coordinate chart $\varphi_n:\IC^{2k}\to\CP^{2k}$ with $\varphi_n(0)=u^0_n$, and we will identify $\tilde{u}^k$ with the map $\varphi_n\circ\tilde{u}^k$ with image in $\IC^{2k}$. We stress that the symplectomorphism $\phi^0_1$ maps the coordinate neighborhood to itself and preserves the norm. For the proof we consider a family of bounded open subset of $\IR^2$ obtained by translations with respect to $s\in\IR$ and all norms are understood after restricting the maps $\tilde{u}^k$ to these bounded open subsets. Using the translations with respect to $s\in\IR$ as well as $\tilde{u}(\cdot,t+1)=\phi^0_{-1}(\tilde{u}(\cdot,t))$ we will show that the $C^{\alpha}$-norm of $\tilde{u}^k$ is bounded uniformly in $k\in\IN$. \\

Since from the lemma we know that the maximum norms of $\del_s\tilde{u}^k$ and $\del_t\tilde{u}^k$ are bounded (uniformly in $k$), we already know that $\|\tilde{u}^k\|_{C^1}$ is bounded. In order to show that $\|\tilde{u}^k\|_{C^{\alpha}}$ is bounded for all $\alpha\in\IN$, we apply the classical elliptic regularity result, together with lemma \ref{approx}. For every $s\in\IR$ we consider the restriction of $\tilde{u}^k$ to the bounded subset $(s-\Delta s, s+\Delta s)\times (-1,1)\subset\IR^2$ for some fixed $\Delta s>0$. Fix some $p>2$ and introduce for every $\alpha\geq 1$ the Sobolev $H^{\alpha,p}$-norm $\|\cdot\|_{\alpha,p}=\|\cdot\|_{H^{\alpha,p}}$; note that here we restrict the map to the chosen bounded open subset. By the well-known Sobolev embedding theorem relating the Sobolev $H^{\alpha,p}$-norms with the $C^{\beta}$-norms for different $\alpha,\beta\in\IN$, note that for all $\beta\leq\alpha-2/p$ we have $$\|\tilde{u}^k\|_{C^{\beta}}\leq c_0\cdot\|\tilde{u}^k\|_{H^{\alpha,p}}\,\,\textrm{with a constant}\,\,c_0>0,$$ where the constant $c_0=c_0(\alpha,p)$ is independent of the dimension of the target space. \\

We now prove by induction that $\|\tilde{u}^k\|_{\alpha,p}$ is bounded for all $\alpha\geq 1$. For the induction start, note that the bound on $\|\tilde{u}^k\|_{C^1}$ implies that $\|\tilde{u}^k\|_{H^{1,p}}$ is bounded; note that this is the point where it is crucial that we first restrict $\tilde{u}^k$ to a bounded open subset. For the induction step, let us assume that $\|\tilde{u}^k\|_{\alpha,p}$ is uniformly bounded in $k$. Note that $\CR\tilde{u}^k=-\varphi_{T_k}(s)\cdot\nabla G^k_t(\tilde{u}^k)$ is equivalent to $\CR\tilde{u}^k=\eta^k$ with $\eta^k=-\varphi_{T_k}(s)\cdot \nabla G^k_t(\tilde{u}^k)$; in particular $\|\eta^k\|_{\alpha,p}$ is bounded if and only if the $H^{\alpha,p}$-norm of $\nabla G^k(\tilde{u}^k)$ is bounded with $\nabla G^k(\tilde{u}^k)(s,t)=\nabla G^k_t(\tilde{u}^k(s,t))$. On the other hand, viewing $\nabla G^k(\tilde{u}^k)$ as composition of the maps $\check{u}^k:(s,t)\mapsto (s,t,\tilde{u}^k(s,t))$ and $\nabla G^k:(s,t,u)\mapsto\nabla G^k_t(u)$, we can use (\cite{MDSa}, appendix B) to deduce that $$\|\nabla G^k(\tilde{u}^k)\|_{H^{\alpha,p}}\leq c_1\cdot \|\nabla G^k\|_{C^{\alpha}}\cdot(1+\|\check{u}^k\|_{L^\infty}^{\alpha-1})(1+\|\check{u}^k\|_{H^{\alpha,p}})$$ with a constant $c_1>0$ which is independent of the dimension of the target space. Note that here we view $\nabla G^k$ as a map from $\IR\times\CP^{2k}$ to $\CP^{2k}$ given by $\nabla G^k(t,u)=\nabla G^k_t(u)$; in particular, the $C^{\alpha}$-norm also contains $t$-derivatives of $t\mapsto \nabla G^k_t$. Since by lemma \ref{approx} we have for all $\alpha\in\IN$ that $\|\nabla G^k\|_{C^{\alpha}}\to\|\nabla G\|_{C^{\alpha}}$ as $k\to\infty$, it follows that $\|\nabla G^k\|_{C^{\alpha}}$ is bounded. Since $\|\tilde{u}^k\|_{\alpha,p}$ is bounded, it follows that the $H^{\alpha,p}$-norm of $\nabla G^k_t(\tilde{u}^k)$ and hence $\|\eta^k\|_{\alpha,p}$ is bounded. In order to complete the induction step, we apply the local regularity for the $\CR$-operator in (\cite{MDSa}, theorem B.3.4) in order to obtain $$\|\tilde{u}^k\|_{\alpha+1,p}\leq c_2 \big(\|\CR\tilde{u}^k\|_{\alpha,p}+\|\tilde{u}^k\|_p\big),$$ where the constant $c_2>0$ in (\cite{MDSa}, theorem B.3.4) is again independent of the dimension of the target space. Together with the boundedness of $\|\CR\tilde{u}^k\|_{\alpha,p}=\|\eta^k\|_{\alpha,p}$ and $\|\tilde{u}^k\|_p\leq\|\tilde{u}^k\|_{\alpha,p}$ this proves that $\|\tilde{u}^k\|_{\alpha+1,p}$ is still bounded. \end{proof}

\section{Normal component and the small divisor problem} 

Recall from the proof of proposition \ref{Liouville} that, if the Floer curve $(\tilde{u},T)\in\IM^{n,k}$ has image in $\CP^{2k}\backslash D_{k,\ell}$ for some $k\geq\ell$, then we can write $\tilde{u}$ as a pair of maps, $$\tilde{u}=(\tilde{u}^{\ell}_{\perp},\tilde{u}^{\ell}):\IR\times\IR\to \IC^{2k-2\ell}\times \CP^{2\ell},$$ where $\tilde{u}^{\ell}$ is the projection onto $\CP^{2\ell}$ and $\tilde{u}^{\ell}_{\perp}$ remembers the normal component. Furthermore we again denote by $\IM^{n,k}$ the connected moduli space containing $(u^0_n,0)$. \\

In order to be able to show that a sequence of Floer curves $\tilde{u}^k$ converge in the $C^{\infty}_{\loc}$-sense to a Floer curve $\tilde{u}:\IR\times\IR\to\IP(\IH)$ as in the main theorem, the key challenge is to be able to control the normal component $\tilde{u}^{k,\ell}_{\perp}$. We emphasize that the proof of the following proposition crucially relies on a deep result from the theory of diophantine approximations obtained using methods from analytic number theory.   
 
\begin{proposition}\label{step1} There exists some $\ell_0\in\IN$ such that for all $k\geq\ell\geq\ell_0$ every Floer curve $(\tilde{u},T)$ in $\IM^{k,n}$ has image in $\CP^{2k}\backslash D_{k,\ell}$. Moreover, we have $$\sup_{k\geq\ell}\|\tilde{u}^{\ell}_{\perp}\|_{C^\alpha}\cdot\ell^\delta\to 0\,\,\textrm{as}\,\,\ell\to\infty\,\,\textrm{for every}\,\,\alpha,\delta\in\IN,$$ where the supremum is taken over all Floer curves $(\tilde{u},T)\in\IM^{k,n}$, $k\geq\ell$. \end{proposition}

In the above proposition the norm refers to the standard Euclidean norm on $\IC^{2k-2\ell}$, but the statement in particular implies that $$\sup_{k\geq\ell} \sup_{(s,t)\in\IR^2} d(\tilde{u}^{\ell}_{\perp}(s,t),0)\cdot\ell^{\delta}\to 0\,\,\textrm{as}\,\,\ell\to\infty$$ for all $\delta\in\IN$, where $d$ denotes the distance with respect to the standard Riemannian metric on $\CP^{2k-2\ell}$ from $0\in\IC^{2k-2\ell}\subset\CP^{2k-2\ell}$. \\

Similar to the proof of proposition \ref{Liouville}, it suffices to show the following \\

\noindent\emph{Claim: For all $\alpha,\delta\in\IN$ it holds that $\sup_{k\geq\ell} \|\tilde{u}^{\ell}_{\perp}\|_{C^\alpha}\cdot\ell^{\delta}\to 0$ as $\ell\to\infty$, where the supremum is taken over all Floer curves $(\tilde{u},T)\in\IM^{k,n}$ with image in $\CP^{2k}\backslash D_{k,\ell}$, $k\geq\ell$.}\\

In order to see that this implies the statement of the proposition, observe  that by the claim we can find some $\ell_0\in\IN$ such for all $\ell\geq\ell_0$ we have $\sup_{\ell} \sup_{(s,t)\in\IR^2} d(\tilde{u}^{\ell}_{\perp}(s,t),0)\leq \pi/4 < d(\CP^{2\ell},D_{k,\ell})$, where the supremum is again taken only over the Floer curves with image in $\CP^{2k}\backslash D_{k,\ell}$. The statement of the proposition then follows, similar as in the proof of proposition \ref{Liouville}, using that $\tilde{u}_{n,0}=u^0_n\in\CP^{2n}\subset\CP^{2\ell}$ for $T=0$ and the fact that the connected moduli space $\IM^{k,n}$ is a connected one-dimensional manifold which is continuous with respect to the underlying $H^{1,p}$-norm and hence also the weaker $C^0$-norm. Indeed, considering the subset of tuples $(\tilde{u},T)$ for which $\tilde{u}$ has distance at most $\pi/4$ from $\CP^{2\ell}$, it follows from continuity that it is closed and, after additionally employing the claim, that it is open as well.\\

\begin{proof}\emph{(of proposition \ref{step1})} We consider the densly defined operator $A:=-i\del_t$ on $$L^2_{\phi^0}(\IR,\IH)=\{u\in L^2(\IR,\IH)=L^2(\IR\times (\IR/2\pi\IZ),\IC): u(t+1,\cdot)=\phi^0_{-1}(u(t,\cdot))\}.$$ It has a complete basis of eigenfunctions $u_{m,p}$ with eigenvalues $\lambda_{m,p}$ given by $$u_{m,p}(t,x)=\exp(imx)\cdot\exp(i(2\pi p-m^2)t)\,\,\textrm{and}\,\,\lambda_{m,p}=2\pi p-m^2,\,m,p\in\IZ.$$ Note that here we use that the sequence of functions $x\mapsto\exp(imx)$, $m\in\IZ$ is a complete basis of eigenfunctions of the time-one map $\phi^0_1$ on $\IH$ with corresponding eigenvalues given by $\exp(im^2)$, $m\in\IZ$. While all eigenvalues are non-zero, we have  $\inf\{|2\pi p-m^2|:\,m,p\in\IZ\}=0$. This is called the \emph{small divisor problem}. On the other hand, it follows from the theory of diophantine approximations that $$\inf_{p\in\IZ} |p\cdot 2\pi-m^2|\,=\,2\pi\cdot q\cdot \inf_{p\in\IZ} \Big|\frac{1}{2\pi}-\frac{p}{q}\Big| \,\geq\, q\cdot \frac{c}{q^r}\,\geq\,c\cdot m^{-14}$$ with $q=m^2$, some constant $c>0$ and $r>0$ denoting the irrationality measure of $(2\pi)^{-1}$ which equals the one of $\pi$ and is known to be smaller than $8$ following \cite{Sa}. On the other hand, by passing from Hilbert space $\IH$ to projective Hilbert space $\IP(\IH)$, we stress that in the canonical coordinates $\varphi_n$ around $u^0_n$ the induced operator $A=-i\del_t$ on $L^2_{\phi^0}(\IR,\IH)$ has eigenfunctions $u^n_{m,p}$ with eigenvalues $\lambda^n_{m,p}$ given by $$u^n_{m,p}(t,x)=\exp(imx)\cdot\exp(i(2\pi p-m^2+n^2)t)\,\,\textrm{and}\,\,\lambda_{m,p}=2\pi p-m^2+n^2$$ for all $m,p\in\IZ$. In order to see this, it suffices to observe that the time-one map $\phi^0_1$ maps $[\hat{u}]\in\IP(\IH)$ with $\hat{u}(n)=1$ to $[\widehat{\phi^0_1(u)}]\in\IP(\IH)$ with $\widehat{\phi^0_1(u)}(n)=\exp(in^2)$. It follows that $$\inf_{p\in\IZ}|\lambda^n_{m,p}|=\inf_{p\in\IZ}|p\cdot 2\pi-m^2+n^2|\,\geq\, c\cdot (m^2-n^2)^{-7}\,\geq\,\frac{1}{2}\cdot c\cdot m^{-14}$$ for $m\geq\ell\geq n$ sufficiently large. In particular, we emphasize that every subsequence of eigenvalues only converges with polynomial speed to $0$.\\

In order to apply this to our situation, choose some $\ell\geq n$ and consider any Floer curve $(\tilde{u},T)$ from $\IM^{k,n}$, $k\geq\ell$. Since $\tilde{u}:\IR^2\to\CP^{2k}\subset\IP(\IH)$ is smooth by lemma \ref{Floer-smooth}, satisfies the periodicity condition $\tilde{u}(\cdot,t+1)=\phi^0_{-1}(\tilde{u}(\cdot,t))$, the Floer equation $0=\CR\tilde{u}+\varphi_T(s)\cdot\nabla G^k_t(\tilde{u})$, and the asymptotic condition $\tilde{u}(s,\cdot)\to u^0_n$, the corresponding normal component $\tilde{u}^{\ell}_{\perp}:\IR^2\to\IC^{2k-2\ell}\subset\IH/\IC^{2\ell+1}$ defines a smooth map from $\IR$ to $L^2_{\phi^0}(\IR,\IC^{2k-2\ell})\subset L^2_{\phi^0}(\IR,\IH/\IC^{2\ell+1})$ and satisfies the Floer equation $\del_s\tilde{u}^{\ell}_{\perp}=A\tilde{u}^{\ell}_{\perp}+\varphi_T(s)\cdot\nabla^{\ell}_{\perp} G^k_t(\tilde{u})$, as well as the asymptotic condition $\tilde{u}^{\ell}_{\perp}(s,t)\to 0$ as $s\to\pm\infty$. Since the normal bundle is trivialized using the canonical coordinates of $\IP(\IH)$ around $u^0_n$, it follows that for the operator $A=-i\del_t$ on the normal bundle the complete basis of eigenfunctions with corresponding eigenvalues is given by $$u^n_{m,p}(t,x)=\exp(imx)\cdot \exp(i(m^2-n^2+2\pi p)t)\,\,\textrm{and}\,\,\lambda^n_{m,p}=m^2-n^2+2\pi p$$ for all $m,p\in\IZ$, $|m|\geq\ell+1$. After composing $\tilde{u}^{\ell}_{\perp}:\IR\to L^2_{\phi^0}(\IR,\IH/\IC^{2\ell+1})$ with the corresponding Fourier transform, we obtain a smooth map $\widehat{\tilde{u}^{\ell}_{\perp}}:\IR\to\ell^2(\IZ\times\IZ,\IC)$, which in turn defines a countable family of smooth maps $$w_{m,p}:\IR\to\IC,\,\,w_{m,p}(s):=\widehat{\tilde{u}^{\ell}_{\perp}(s)}(m,p)\,\,\textrm{for all}\,\,m,p\in\IZ,\, |m|\geq\ell+1$$ satisfying $$
w'_{m,p}(s)=\lambda^n_{m,p}w_{m,p}(s)+f_{m,p}(s),\,\,\textrm{and}\,\,w_{m,p}(s)\to 0\,\, \textrm{as}\,\,s\to\pm\infty$$
with $f_{m,p}(s):=\widehat{\nabla^{\ell}_{\perp} G^k(\tilde{u})(s)}(m,p)$. Here we view $\nabla^{\ell}_{\perp} G^k(\tilde{u})(s)$ as an element of $L^2_{\phi^0}(\IR,\IH/\IC^{2\ell+1})$ by setting $$\nabla^{\ell}_{\perp} G^k(\tilde{u})(s)(t):=\varphi_T(s)\cdot\nabla^{\ell}_{\perp} G^k_t(\tilde{u}(s,t)).$$ Then by combining lemma \ref{exp} with proposition \ref{no-bubbling} we know that $\nabla^{\ell}_{\perp} G^k(\tilde{u})(s)$ is smooth with respect to the time and the space coordinate with uniformly bounded derivatives, which is equivalent to the fact that the Fourier coefficients are decaying exponentially fast, i.e., for every $\alpha,\delta\in\IN$ there exists $c_{\alpha,\delta}>0$ with $$|f_{m,p}(s)|= \big|\widehat{\nabla^{\ell}_{\perp} G^k(\tilde{u})(s)}(m,p)\big|< c_{\alpha,\delta}\cdot |m|^{-\delta}\cdot |p|^{-\alpha}\,\,\textrm{for all}\,\,m,p\in\IZ,\, |m|\geq\ell+1.$$ Note that, since $\nabla^{\ell}_{\perp} G^k(\tilde{u})(s)$ depends on the $t$-coordinate in two different ways, we indeed also need to use that by proposition \ref{no-bubbling} all derivatives of $\tilde{u}$ with respect to $t$ are uniformly bounded for all $(\tilde{u},T)$ from $\IM^{k,n}$, $k\geq n$; furthermore we assume that the perturbation $\nu^k_t$ in $F^k_t:=F^{k,\nu}_t=F_t\circ\pi_k+\nu^k_t$ is chosen to decay exponentially fast as $k\to\infty$. Together with 
$$\inf_{p\in\IZ}|\lambda^n_{m,p}|=\inf_{p\in\IZ}|m^2-n^2+p\cdot 2\pi|\,\geq\, c\cdot (m^2-n^2)^{-7}\,\geq\,\frac{1}{2}\cdot c\cdot m^{-14}$$ for $m\geq\ell\geq n$ sufficiently large, using lemma \ref{technical} (which can be found after the end of the proof) this implies that for all $\alpha,\delta\in\IN$ we have $$\big|\widehat{\tilde{u}^{\ell}_{\perp}(s)}(m,p)\big|=|w_{m,p}(s)|\leq C_{\delta}\cdot |m|^{14-\delta}\cdot |p|^{-\alpha}\,\,\textrm{with}\,\, C_{\alpha,\delta}=2\sqrt{2}\cdot c_{\alpha,\delta}/c>0$$ for all $s\in\IR$ and $m,p\in\IZ$ with $|m|\geq\ell$ sufficiently large. In particular, note that the constants are independent of the chosen Floer curve $(\tilde{u},T)\in\IM^{k,n}$, $k\geq \ell$. When $\widehat{\nabla^{\ell}_{\perp} G^k(\tilde{u})}=0$ and hence $f_{m,p}=0$, then $w_{m,p}=0$ for all $m,p\in\IZ$,  $|m|\geq\ell+1$, i.e., there is only the trivial solution, which provides an alternative proof for proposition \ref{Liouville}.  Since $\delta\in\IN$ is chosen arbitrarily, it follows from   
$$|\del_t^{\alpha}\tilde{u}^{\ell}_{\perp}(s,t)|^2\,\leq\,\frac{1}{\sqrt{2\pi}}\sum_{|m|=\ell+1}^k\Big(\sum_{p=-\infty}^{+\infty} |\widehat{\tilde{u}^{\ell}_{\perp}(s)}(m,p)|\cdot |p|^\alpha\Big)^2\,\,\textrm{for all}\,\, (s,t)\in\IR^2$$ that we have $$\sup_{k\geq\ell} \|\del_t^{\alpha}\tilde{u}^{\ell}_{\perp}\|_{C^0}\cdot\ell^{\delta}\to 0\,\,\textrm{as}\,\,\ell\to\infty\,\,\textrm{for all}\,\,\alpha,\delta\in\IN.$$  Together with $\del_s\tilde{u}^{\ell}_{\perp}=-i\del_t\tilde{u}^{\ell}_{\perp}+\varphi_T(s)\cdot\nabla^{\ell}_{\perp} G^k_t(\tilde{u})$ and $\sup_{k\geq\ell} \|\nabla^{\ell}_{\perp} G^k_t(\tilde{u})\|_{C^\alpha}\cdot\ell^{\delta}\to 0$ as $\ell\to 0$ for all $\alpha,\delta\in\IN$, it follows that we also can control the $s$-derivatives and hence have $$\sup_{k\geq\ell} \|\tilde{u}^{\ell}_{\perp}\|_{C^\alpha}\cdot\ell^{\delta}\to 0\,\,\textrm{as}\,\,\ell\to\infty\,\,\textrm{for all}\,\,\alpha,\delta\in\IN.$$ Note that in all cases the supremum is taken over all Floer curves $(\tilde{u},T)\in\IM^{k,n}$ with image in $\CP^{2k}\backslash D_{k,\ell}$, $k\geq\ell$. \end{proof}

As mentioned in the proof of proposition \ref{step1}, we finish this section by giving a proof of the following elementary 

\begin{lemma}\label{technical} Let $w=w_R+i w_I:\IR\to\IC$ be a smooth solution of $w'(s) = \lambda w(s) + f(s)$ with $w(s)\to 0$ as $s\to\pm\infty$. If $\lambda\in\IR$ and there exists $c>0$ such that $|f(s)|<c$ for all $s\in\IR$, then $|w(s)|\leq\sqrt{2}\cdot c/|\lambda|$ for all $s\in\IR$. \end{lemma}

\begin{proof} The idea of the proof is that, provided $w(s)$ is too large in norm, then the exponential growth in the positive (if $\lambda>0$) or negative (if $\lambda<0$) $s$-direction can no longer be damped by the nonlinearity $f(s)$ in order to achieve $w(s)\to 0$ as $s\to\pm\infty$. To the contrary assume that there exists $s_0\in\IR$ with $|w(s_0)|>\sqrt{2}\cdot c/|\lambda|$. Assume without loss of generality that $|w_R(s_0)|\geq |w_I(s_0)|$ and $w_R(s_0)>0$ and $\lambda>0$; the case when one or both are negative will lead to obvious changes in the proof. Since $w_R(s_0)>c/\lambda$ and $w_R(s)\to 0$ as $s\to +\infty$, it follows from the intermediate value theorem that we can find $s_0<s_1<+\infty$ with $w_R(s_1)=c/\lambda$ and $w_R(s)> c/\lambda$ for all $s_0<s<s_1$. Since $w_R(s_1)<w_R(s_0)$, using the mean value theorem we find $s_0<s_2<s_1$ with $w'_R(s_2)<0$. On the other hand, since $s_0<s_2<s_1$ we have  $w_R(s_2)> c/\lambda$ and hence $\lambda w_R(s_2)>c$, which together with $|f_R(s)|\leq |f(s)| < c$ implies $\lambda w_R(s_2)+f_R(s_2)>0$, providing the required contradiction. \end{proof}
\section{Completing the proof}

We show now that the sequence of Floer curves $\tilde{u}^k=\tilde{u}^k_{n,T_k}:\IR\times\IR\to\CP^{2k}\subset\IP(\IH)$, $k\geq n$ converges in the $C^{\infty}_{\loc}$-sense to a Floer curve $\tilde{u}:\IR\times\IR\to\IP(\IH)$ as in the main theorem, possibly after passing to a suitable subsequence. As in the finite-dimensional case, the idea is to make use of elliptic bootstrapping to find a limit in the $C^{\infty}_{\loc}$-sense. While we have already proven in proposition \ref{no-bubbling}, using bubbling-off analysis in $\IP(\IH)$, that all the derivatives of $\tilde{u}^k$ are uniformly bounded as $k$ converges to infinity, note that this is not sufficient to establish the existence of a $C^{\infty}_{\loc}$-limit due to the non-compactness of $\IP(\IH)$. On the other hand, we show below that the result in proposition \ref{step1} about the normal component, proven using the diophantine approximation result, provides us with the missing piece. 

\begin{lemma}\label{Floer-smooth} After passing to a suitable subsequence, the sequence of Floer curves $\tilde{u}^k=\tilde{u}^k_{n,T_k}:\IR\times\IR\to\CP^{2k}\subset\IP(\IH)$ with $T_k=k$ converges in the $C^{\infty}_{\loc}$-sense to a smooth map $\tilde{u}=\tilde{u}_n:\IR\times\IR\to\IP(\IH)$ with $\tilde{u}_n(\cdot,t+1)=\phi^0_{-1}(\tilde{u}_n(\cdot,t))$ satisfying the Floer equation $$0\,=\,\CR\tilde{u}_n\;+\;\varphi(s)\cdot\nabla G_t(\tilde{u}_n),$$ where $\CR=\del_s+i\del_t$ denotes the standard Cauchy-Riemann operator and $\varphi$ is a smooth cut-off function with $\varphi(s)=0$ for $s\leq -1$ and $\varphi(s)=1$ for $s\geq 0$. \end{lemma}

\begin{proof} By using a diagonal sequence argument we know that, after passing to a subsequence, we may assume that the sequence of maps $\tilde{u}^{k,\ell}=\pi_\ell\circ\tilde{u}^k:\IR\times\IR\to\CP^{2\ell}$ converges in the $C^{\infty}_{\loc}$-sense to a smooth map $\tilde{u}^{\ell}: \IR\times\IR\to\CP^{2\ell}$ as $k\to\infty$ for all $\ell\geq n$ \emph{simultaneously}. Indeed we have already shown in chapter $6$ that we locally have bounded $H^{\alpha,p}$-norms for all $\alpha\in\IN$ and hence, after passing to a diagonal subsequence, local $H^{\alpha,p}$-convergence for all $\alpha\in\IN$. Note that here we crucially make use of the fact that, for fixed $\ell$, the maps $\tilde{u}^{k,\ell}$ have image in the same \emph{finite}-dimensional manifold $\CP^{2\ell}$ as $k\to\infty$.  \\

Now fix $\alpha\in\IN$ and restrict all maps to a given bounded open subset. For given $\eps>0$, we find $\ell\geq\ell_0\geq n$ such that $\sup_{k\geq\ell}\|\tilde{u}^{k,\ell}_{\perp}\|_{C^{\alpha}}<\eps/3$. Since $\tilde{u}^{k,\ell}\to\tilde{u}^{\ell}$ for this given $\ell$, we find $k_0\geq\ell$ such that $\|\tilde{u}^{k,\ell}-\tilde{u}^{k',\ell}\|_{C^{\alpha}}<\eps/3$ for all $k,k'\geq k_0$. But together this gives $$\|\tilde{u}^k-\tilde{u}^{k'}\|_{C^{\alpha}}\,\leq\,\|\tilde{u}^{k,\ell}_{\perp}\|_{C^{\alpha}} \;+\;\|\tilde{u}^{k,\ell}-\tilde{u}^{k',\ell}\|_{C^{\alpha}} \;+\;\|\tilde{u}^{k',\ell}_{\perp}\|_{C^{\alpha}}\,<\,\eps.$$ Here we assume without loss of generality that, after restricting to the bounded open subset, the Floer curve lies in the coordinate neighborhood around $u^0_n$. \end{proof}       

With this we can now finish the proof of the main theorem.

\begin{proof}\emph{(of the main theorem)} For every $n\in\IN$ it remains to be shown that the Floer curve $\tilde{u}=\tilde{u}_n:\IR\times\IR\to\IP(\IH)$ connects the fixed point $u^0_n$ of the free Schr\"odinger equation with a fixed point $u^1_n$ of $\phi_1$ and that $u^1_n\neq u^1_m$ for $m\neq n$. Fixing $n\in\IN$, for the first statement we show that there exist sequences $(s_{\gamma}^{\pm})$ of positive and negative real numbers, $s_{\gamma}^{\pm}\to\pm\infty$ with $\tilde{u}_n(s_{\gamma}^+,\cdot)\to u^1_n$ and $\tilde{u}_n(s_{\gamma}^-,\cdot)\to u^0_n$ as $\gamma\to\infty$. In order to see the latter, note that as in the finite-dimensional case, see the proof of proposition \ref{finite-dim}, for every $\gamma\in\IN$ there exists $\gamma\leq |s_{\gamma}^{\pm}|\leq 2\gamma$ such that $$\int_0^1 |\del_t\tilde{u}(s_{\gamma}^{\pm},t) - \varphi(s_{\gamma}^{\pm}) X^G_t(\tilde{u}(s_{\gamma}^{\pm},t))|^2\;dt\,<\,\frac{\pi}{\gamma}.$$ In order to show that a subsequence converges, note that we can again write the Floer map as a tuple, $$\tilde{u}=(\tilde{u}^{\ell}_{\perp},\tilde{u}^{\ell}): \IR\times\IR\to (\IH/\IC^{2\ell})\times\CP^{2\ell},$$ where $\|\tilde{u}^{\ell}_{\perp}\|_{\infty}\to 0$ as $\ell\to\infty$ due to proposition \ref{step1}. By compactness of $\CP^{2\ell}$ we know, possibly after passing to a subsequence, that for every $\ell\geq n$ the sequences $\tilde{u}^{\ell}(s_{\gamma}^{\pm},0)$ converge. Together with $\tilde{u}^{\ell}_{\perp}(s_{\gamma}^{\pm},0)\to 0$ as $\ell\to\infty$ this now proves that $\tilde{u}(s_{\gamma}^+,0)$ converges to a fixed point $u^1_n$ of $\phi_1$ and $\tilde{u}(s_{\gamma}^-,0)$ converges to the fixed point $u^0_n$ of $\phi^0_1$. In order to see that $\tilde{u}(s_{\gamma}^-,\cdot)$ indeed converges to the fixed point $u^0_m$ of $\phi^0_1$ with $m=n$, recall that breaking-off of cylinders for the free Hamiltonian $H^0$ can be excluded by index reasons. \\

In order to see that the resulting fixed points $u^1_m$, $u^1_n$ for $H_t=H^0+F_t$ are different if $m\neq n$ as well, as in the finite-dimensional case it suffices to refer to Floer's proof in \cite{Fl} of the degenerate Arnold conjecture for $\CP^{2k}$ using the cup action on Floer cohomology for finite-dimensional Hamiltonians. Since Floer cohomology can only be defined for Hamiltonians with nondegenerate one-periodic orbits, we again consider the sequence of perturbed finite-dimensional Hamiltonians $H^k_t=H^0+F^k_t$ with $H^k_t\to H_t$. For the rest the argument is completely similar to the one sketched in the proof of proposition \ref{finite-dim}, see \cite{Fl} for further details.  \end{proof}

Finally, we show how our findings imply the existence of infinitely many different strong solutions of the original Hamiltonian PDE.

\begin{proof}\emph{(of the proposition \ref{strong})} Recall from the proof of proposition \ref{step1} that for the sequence of Floer curves $\tilde{u}^k:\IR\times\IR\to\CP^{2k}$ converging to the Floer curve $\tilde{u}:\IR\times\IR\to\IP(\IH)$ we know that for all $\ell\geq n$ sufficiently large and all $\alpha,\delta\in\IN$ we have $\sup_{k\geq\ell} \|\tilde{u}^{k,\ell}_{\perp}\|_{C^\alpha}\cdot\ell^{\delta}\to 0$ as $\ell\to\infty$. It follows that the corresponding decay rate continues to hold for the limiting curve $\tilde{u}$, that is, for all $\alpha,\delta\in\IN$ we have $\sup_{k\geq\ell} \|\tilde{u}^{\ell}_{\perp}\|_{C^\alpha}\cdot\ell^{\delta}\to 0$ as $\ell\to\infty$. This implies that $\tilde{u}$ and hence also $u^1_n$ sit in $\IP(C^{\infty}(S^1,\IC))\subset\IP(L^2(S^1,\IC))$ and it follows for all $\alpha,\delta\in\IN$ that $\widehat{u^1_n}(m,p)\cdot |m|^\delta\cdot |p|^\alpha\to 0$ as $|m|,|p|\to\infty$ which in turn is equivalent to the fact that every point $(t,x)\in\IR^2$ all partial derivatives $\del_x^{\delta}\del_t^{\alpha} u^1_n(t,x)$, $(\alpha,\delta)\in\IN^2$ exist. \\

In order to see that the path $u^1_n:\IR\to\IP(L^2(S^1,\IC))$ provides us with a strong solution of the nonlinear Schr\"odinger equation of convolution type as stated in the corollary, we just need to choose a lift of $u^1_n$ to a path in $\IS(L^2(S^1,\IC))\subset L^2(S^1,\IC)$ which is a weak solution of the nonlinear Schr\"odinger equation. This provides us with a map $u:\IR^2\to\IC$ which automatically satisfies the periodicity condition $|u(t+1,x)|=|u(t,x)|=|u(t,x+2\pi)|$ for all $(t,x)\in\IR^2$. In order to prove that $u$ is a strong solution of the nonlinear Schr\"odinger equation, it suffices to observe that $u$ is smooth. \end{proof}

\end{document}